\documentclass[12pt]{article}

\usepackage{amsmath}
\usepackage{amsthm}
\usepackage{amssymb}
\usepackage{breqn}
\usepackage{graphicx}
\usepackage{float}
\usepackage{hyperref}

\theoremstyle{plain}
\newtheorem{theorem}{Theorem}[section]
\newtheorem{definition}[theorem]{Definition}

\newtheorem{property}[theorem]{Property}
\newtheorem{lemma}[theorem]{Lemma}
\newtheorem{corollary}[theorem]{Corollary}

\newcommand{\longueur}[1]{\ensuremath{\lvert #1 \rvert}}
\newcommand{\generateur}[1]{#1_0} \newcommand{\prefix}[1]{#1_1} \newcommand{\suffix}[1]{#1_2} 

\newcommand{\Useul}[2][u]{\ensuremath{\generateur{#1}{}^{{#2}_1}\prefix{#1}\generateur{#1}{}^{{#2}_2}}}
\newcommand{\Udeux}[2]{\Useul[#1]{#2}.\Useul[#1]{#2}}
\newcommand{\UU}[1][e]{\Udeux{u}{#1}}

\newcommand{\dsuu}{\ensuremath{\mathcal{U} }}
\newcommand{\dsuv}{\ensuremath{\mathcal{V} }}
\newcommand{\dsuw}{\ensuremath{\mathcal{W} }}
\newcommand{\conj}[1]{\tilde{#1}}

\DeclareMathOperator{\lcp}{lcp}
\DeclareMathOperator{\lcs}{lcs}
\DeclareMathOperator{\classe}{Cl}
\usepackage{color}

\begin{document}
\bibliographystyle{plain}

\title{A proof that a word of length $n$ has less than $1.5n$ distinct squares}
\author{Adrien Thierry}
\maketitle


\begin{abstract}
We are interested in the maximal number of distinct squares in a word. This problem was introduced by Fraenkel and Simpson, who presented a bound of $2n$ for a word of length $n$, and conjectured that the bound was less than $n$. Being that the problem is on repetitions, their solution relies on Fine and Wilf's Periodicity Lemma. Ilie then refined their result and presented a bound of $2n - \Theta (\log n)$. Lam used an induction to get a bound of $\frac{95}{48}n$. Deza, Franek and Thierry achieved a bound of $\frac{11}{6}n$ through a combinatorial approach. \\
Using the properties of the core of the interrupt presented by Thierry, we refined here the combinatorial structures exhibited by Deza, Franek and Thierry to offer a bound of $\frac{3}{2}n$.
\end{abstract}



\section{Introduction}

Repetitions are one of the most fundamental and well studied phenomenon in sequences. In combinatorics on words, multiplication is defined as concatenation, and a square is the most elementary form of repetition. If the maximal number of squares in a word of length $n$ is $n^2/4$, the number of distinct squares is less than $n/2$ for a word constructed over a unitary alphabet but can be in the order of $n$ for binary alphabets (as presented in \cite{FS1}), which asks the question: "What is the maximal number of distinct squares in a word of length $n$ ?". Fraenkel and Simpson noted and proved in \cite{FS1} that no more than two squares can have their last occurence starting at the same position using Crochemore and Rytter three squares lemma \cite{CR95}. This fact leads to a bound of $2n$ distinct squares in a word of length $n$ (a proof can be found in \cite{lothaire2}). But if two squares have their last occurence starting at the same position (we will call them FS double squares), a restrictive condition, they have to obey a certain structure. As in \cite{DFT1}, we will use those structures to relate the different FS double squares of a word to one another. Some of these relations are restrictive and will limit the number of FS double squares involved. Other will only give informations on the relative positions of the two FS double squares. We will prove that the relations defined are the only one possible, hence that we have an exhaustive list of the different FS double squares of a word. Finally, we will offer an induction that, using all the results that will have been presented, will prove that the number of FS double squares of a word is less than $\frac{n}{2}$. The bound for the number of distinct squares is a corollary. \\

The major improvements from the previous article, \cite{DFT1}, are the analysis of the $\gamma$-mates and the amortization of the $\epsilon$-mates. Using the core of the interrupt of \cite{moa1} instead of the inversion factor, we could redefine $\epsilon$-mates. The presence of $\beta$-mates prevents the existence of a certain number of (new) $\epsilon$'s.
This amortization allows the induction principle to rely on the tail of $u$ and $v$ only (and not on their gap anymore). 
The gap represent how far $v$ starts after $u$. the tail how fat it ends after $u$ has ended ($ut = gv$, for $g$ the gap ant $t$ the tail). 
There are examples of peculiar double squares that start and end not far enough from u: the (new) $\epsilon$-mates, and we had to make concessions on the tightness of the result.
By incorporating the $\epsilon$-mates in the family, we get rid of that problem, and by analyzing the relationship between $\beta$ and $\epsilon$-mates, the result is tightened. \\


\section{Notations}

An \emph{alphabet} $A$ is a finite set. 
 The elements of $A$ are called \emph{letters}. \\
 A  \emph{word} $w$ is a sequence of letters of length $n$, we note its length $\longueur{w} =n$, which can also be presented under the form of an array $w[1,...,n]$.
  We define \emph{multiplication} as concatenation. 
In a traditional fashion, we define the \emph{$n^{th}$ power} of a word $w$ as $n$ times the multiplication of $w$ with itself.
 A word $x$ is \emph{primitive} if $x$ cannot be expressed as a non-trivial power of another word $x'$. The primitive root of a word $x$ is the primitive word $y$ such that $x=y^n$ for an integer $n$ (uniqueness by Fine and Wilf, \cite{finewilf}).\\
If $x = x_1x_2x_3$ for non-empty words $x_1, x_2$ and $x_3$, then $x_1$ is a \emph{prefix} of $x$, $x_2$ is a \emph{factor} of $x$, and $x_3$ is a \emph{suffix} of $x$ (if both the prefix and the suffix are non empty, we refer to them as proper). The starting position of $x_2$ is $s(x_2) = \lvert x_1 \rvert +1$, its ending position $end(x_2)=\lvert x_1 \rvert +\lvert x_2 \rvert$. A word $\tilde{x}$ is a \emph{conjugate} of $x$ if $x=x_1x_2$ and $\tilde{x}=x_2x_1$ for non-empty words $x_1$ and $x_2$. The \emph{conjugacy class} of $w$, denoted $\classe (w)$, is the set of the conjugates of $w$. \\
 A factor $x, \lvert x \rvert =n$ of $w$ has \emph{period} $p\leq\frac{n}{2}$ if $x[i]=x[i+\lvert p \rvert], \forall i \in [1,...n-\lvert p\rvert]$.\\
 A factor $w[i+p... j+p]$ of $w$ is a cyclic shift of of $w[i...j]$ by $p$ positions if $w[i...j+p]$ has period $j-i+1$ if $p\geq 0 $, or if $w[i+p...j]$ has period $j-i+1$ if $p<0$.  \\
The \emph{longest common prefix} and \emph{suffix} of $x$ and $y$ are  denoted $lcp(x,y)$ and $lcs(x,y)$ respectively (note that $\lcs (x,y)$ and $\lcp (x,y)$ are words).


\section{Tools}

We first present the discrete version of Fine and Wilf's Periodicity Lemma which answers the question: "How long does a periodic sequence of periods $m$ and $n$ has to be to ensure that it also has period $\gcd(m,n)$, their greatest common divisor ?".

\begin{lemma}\label{fw} [Fine and Wilf's Periodicity Lemma, \cite{finewilf}]
Let $w$ be a word having periods $x$ and $y$. If $\lvert w \rvert \geq x + y - \gcd (x, y)$, then $w$ has period $\gcd (x, y)$. 
\end{lemma}

\begin{corollary}\label{sp}[Synchronization principle]
If $u$ is primitive, $\forall \conj{u} \in \classe(u), u \neq \conj{u}$.
\end{corollary}

We now present the different results, of Crochemore, Deza, Fraenkel, Franek, Lam, Rytter, Simpsons and Thierry used in the different proofs.


\begin{lemma}\label{crts}[Crochemore and Rytter's Three Squares Lemma, \cite{CR95}]
Let $u_1^2$, $u_2^2$, $u_3^2$ be three prefixes of a word $x$ with $u_1$ primitive and $\longueur{u_1} < \longueur{u_2} < \longueur{u_3}$. Then $\longueur{u_1} + \longueur{u_2} \leq \longueur{u_3}$.
\end{lemma}

\begin{lemma}\label{fsdsb}[Fraenkel and Simpson's Double Square Lemma, \cite{FS1}]
At most two squares can have their last occurrence starting at the same position.
\end{lemma}

A short proof of \ref{fsdsb} is provided by Ilie in \cite{Ilie1}.

\begin{definition}
In a word $w$ a FS double square, often denoted \dsuu, is a set of two square factors having their last occurrences starting at the same position in $w$. The number of FS double squares of $w$ is $\delta(w)$.
\end{definition}

\begin{lemma}\label{lams}[Factorization, \cite{DFT1}]
Let \dsuu\ be a FS double square. There exists a unique primitive word $\generateur{u}$, a proper prefix $\prefix{u}$ of $\generateur{u} = \prefix{u}\suffix{u}$ and two integers $e_1, e_2, e_1 \geq e_2 \geq 1$ such that $u=\generateur{u}^{e_1}\prefix{u}, U=\generateur{u}^{e_1}\prefix{u}\generateur{u}^{e_2}$.
\end{lemma}

This factorization is also used in Lam, \cite{Lam}.

\begin{figure}[H]
\[ w = \rlap{$\overbrace{\phantom{aaba}}^u\overbrace{\phantom{aaba}}^u$}\rlap{$\underbrace{\phantom{aabaaab}}_U\underbrace{\phantom{aabaaab}}_U$}aabaaabaabaaab \]
\[ \generateur{u}=aab, \prefix{u}=a, e_1=e_2=1\]
  \caption{\dsuu\ is a FS double square.}
\end{figure}

\begin{definition}\label{fsds}
Let \dsuu\ be a FS double square. Then $\dsuu = \generateur{u}^{e_1}\prefix{u}\generateur{u}^{e_2}\generateur{u}^{e_1}\prefix{u}\generateur{u}^{e_2}$ for a primitive word $\generateur{u}$ and a proper prefix $\prefix{u}$ of $\generateur{u}=\prefix{u}\suffix{u}$. We call $\generateur{u}^{e_1}\prefix{u}\generateur{u}^{e_2}\generateur{u}^{e_1}\prefix{u}\generateur{u}^{e_2}$ the canonical factorization of \dsuu, $u$ its short repetition, $U$ being the long one, $e_1, e_2$ its first and second exponent respectively. We sometimes refer to the first occurence of $u$, starting at position $1$ as $u_{[1]}$, the one at position $\longueur{u}+1$ as $u_{[2]}$ and the one at position $\longueur{U}+1$ as $u_{[3]}$.
\end{definition}

\begin{lemma}\label{dft}[Deza, Franek and Thierry's Longest Common Border Lemma, \cite{DFT1}]
If $x$ is primitive, for all of its conjugates $\conj{x}$, $\longueur{\lcp (x,\conj{x})} + \longueur{\lcs(x, \conj{x})} \leq \longueur{x}- 2$.
\end{lemma}

\begin{definition}\label{tci}
For \dsuu\ a FS double square, write the factors $\suffix{u}\prefix{u}\prefix{u}\suffix{u}$ of \dsuu\ as $\suffix{u}\prefix{u}\prefix{u}\suffix{u} = pr'_pr'r'_sspr_prr_ss$ for $p = \lcp (\prefix{u}\suffix{u}, \suffix{u}\prefix{u})$ and $s = \lcs (\prefix{u}\suffix{u}, \suffix{u}\prefix{u})$, the letters $r_p \neq r'_p$ and $r_s \neq r'_s$ and the (possibly empty) words $r$ and $r'$. The two factors $r'_sspr_p$ are the cores of the interrupt of \dsuu. In \dsuu\ they start at positions $N_1(\dsuu) = \lvert u \rvert - \lcs (\prefix{u}\suffix{u}, \suffix{u}\prefix{u}) $ and $N_2(\dsuu) = \lvert U \rvert  + \lvert u \rvert - \lcs (\prefix{u}\suffix{u}, \suffix{u}\prefix{u}) $.
\end{definition}

\begin{property}\label{ifc}[Thierry's Core of the Interrupt, \cite{moa1}]
Let \dsuu\ be a FS double square. The factors $w_1$ and $w_2$ of length $\longueur{\generateur{u}}$ that end and start with the core of the interrupt of \dsuu\ are not conjugates of $\generateur{u}$.
\end{property}

\begin{figure}[H]

\[ w = aab\rlap{${\overbrace{\phantom{b\mathbf{aaa}}}^{w_1}}$}b\underbrace{\mathbf{aaa}b}_{w_2}baabb\mathbf{aaa}bb \]
\[ \generateur{u} = aabb, \prefix{u} = a, \suffix{u}=abb, e_1=e_2=1 \]
  \caption{The cores of the interrupt of \dsuu\ (in bold).}
\end{figure}

Id est, all the occurrences of $w_1$ and $w_2$ in \dsuu\ contain its core of the interrupt. We will use those two factors as two notches in a FS double square that force the alignment of other repeated factors.


\section{Mates of \dsuu}\label{section3}
We present here how different FS double squares can relate to each other.

\begin{definition}
Let $w$ be a word starting with a FS double square $\dsuu=\UU$. A \emph{$u$-close} double square is a FS double square that starts within the first occurrence of $u$, $u_{[1]}$.
\end{definition}

\begin{definition}
[$\alpha$-mates, \cite{DFT1}]
An $\alpha$-mate of \dsuu\ is a $u$-close double square \dsuv\ where $v$ is a cyclic shift of $u$ and $V$ is a cyclic shift of $U$.
\end{definition}

\begin{figure}[H]
\begin{align*}
\dsuu = (a&ab) (aab) a (aab) . (aab) (aab) a (aab) \\
\dsuv = (&aba) (aba) a (aba) . (aba) (aba) a (aba) \\ \text{and }
w = \rlap{$\overbrace{\phantom{aabaabaaabaabaabaaab}}^{\dsuu}$} a & \underbrace{abaabaaabaabaabaaaba}_{\dsuv}
\end{align*}
  \caption{\dsuv\ is a $\alpha$-mate of \dsuu.}
\end{figure}

As proved in \cite{DFT1}, an $\alpha$-mate of \dsuu\ has the factorization: $\dsuv=\conj{\generateur{u}}^{e_1}\conj{\prefix{u}}\conj{\generateur{u}}^{e_2}\conj{\generateur{u}}^{e_1}\conj{\prefix{u}}\conj{\generateur{u}}^{e_2}$ for $\conj{\generateur{u}}$ a conjugate of $\generateur{u}$, and $\conj{\prefix{u}}$ the prefix of $\conj{\generateur{u}}$ of length $\longueur{\prefix{u}}$.

\begin{definition} [$\beta$-mates, \cite{DFT1}]
A $\beta$-mate is a $u$-close double square \dsuv\ that is not an $\alpha$-mate and where $V$ is a cyclic shift of $U$.
\end{definition}

\begin{figure}[H]
\begin{align*}
\dsuu = (aab) &(aab) (aab) a (aab) . (aab) (aab) (aab) a (aab) \\
\dsuv = &(aab) (aab) a (aab) (aab) . (aab) (aab) a (aab) (aab) \\ \text{and }
w = \rlap{$\overbrace{\phantom{aabaabaabaaabaabaabaabaaab}}^{\dsuu}$} aab & \underbrace{aabaabaaabaabaabaabaaabaab}_{\dsuv}
\end{align*}
  \caption{\dsuv\ is a $\beta$-mate of \dsuu.}
\end{figure}

As proved in \cite{DFT1}, a $\beta$-mate of \dsuu\ has the factorization: \[\dsuv=\conj{\generateur{u}}^{e_1-t}\conj{\prefix{u}}\conj{\generateur{u}}^{e_2+t}\conj{\generateur{u}}^{e_1-t}\conj{\prefix{u}}\conj{\generateur{u}}^{e_2+t},\] for $\conj{\generateur{u}}$, a conjugate of $\generateur{u}$, $\conj{\prefix{u}}$ the prefix of $\conj{\generateur{u}}$ of length $\longueur{\prefix{u}}$, and an integer $t, 1 \leq t \leq \lfloor \frac{e_1-e_2}{2} \rfloor$ (as the second exponent of a double square has to be smaller than the first one).

\begin{definition} [$\gamma$-mates, \cite{DFT1}]
A $\gamma$-mate is a $u$-close double square \dsuv\ where $v$ is a cyclic shift of $U$.
\end{definition}

The factorization of a $\gamma$-mate of a FS double square \dsuu\ is more complicated and will be presented in the proof on property \ref{gammas2}.

\begin{definition} [$\delta$-mates, \cite{DFT1}]
A $\delta$-mate is a $u$-close double square \dsuv\ where $\lvert v \rvert > \lvert U \rvert$.
\end{definition}

Before introducing the new mates of \dsuu, we require one definition.

\begin{definition} [trace]
Let $w = UUr$ be a word starting with a FS double square \dsuu. We define $tr(\dsuu, w)$, the trace of \dsuu\ in $w$ as $tr(\dsuu, w) = \max \{t \in \mathbb{N}, \exists r', r = \generateur{u}^tr'\}$.
\end{definition}

\begin{definition} [$\epsilon$-mates]
Let $w= UU\generateur{u}^{tr(\dsuu, w)}r$, for a possibly empty word $r$, be a word starting with a FS double square \dsuu. An $\epsilon$-mate of \dsuu\ is a FS double square \dsuv\ where $v_{[1]}$ starts after the end of $u_{[1]}$, ends before $\longueur{UU} + tr(\dsuu, w)\lvert \generateur{u} \rvert + \lvert \lcp (\generateur{u}, r) \rvert + 1$, and where $\lvert v \rvert \leq \lvert \generateur{u} \rvert$.
\end{definition}

\begin{definition} [$\eta$-mates]\label{etas1}
Let $w=UUr$ be a word starting with a FS double square \dsuu. An $\eta$-mate of \dsuu\ is a FS double square \dsuv\ where $v_{[1]}$ starts after the end of $u_{[1]}$, ends before $\longueur{u\generateur{u}^{e_1+e_2-1}p}$ for $p=\lcp (\prefix{u}\suffix{u}, \suffix{u}\prefix{u})$ and where $\longueur{v}=\longueur{\generateur{u}^n}$ for some $n\geq2$.
\end{definition}

\begin{definition} [$\zeta$-mates]\label{zetamates}
Let $w= UUr$ for a possibly empty word $r$, be a word starting with a FS double square \dsuu. A $\zeta$-mate is a FS double square \dsuv\ that starts after the end of $u_{[1]}$ and that is neither an $\epsilon$-mate nor an $\eta$-mate of \dsuu.
\end{definition}

A $\zeta$-mate of \dsuu\ is therefore a FS double square \dsuv\ where $v_{[1]}$ starts after the end of $u_{[1]}$, where $\forall n \geq 2, \longueur{v}\neq n\longueur{\generateur{u}}$ if $v_{[1]}$ ends before $u\generateur{u}^{e_1+e_2-1}p$ for $p=\lcp (\prefix{u}\suffix{u}, \suffix{u}\prefix{u})$, and where $\longueur{v} > \longueur{\generateur{u}}$ if $v_{[1]}$ ends before $\longueur{UU} + tr(\dsuu, w)\lvert \generateur{u} \rvert + \lvert \lcp (\generateur{u}, r) \rvert + 1$.


\section{Exhaustivity}

\begin{lemma}\label{exhaust}
Let $w$ be a word starting with a FS double square \dsuu. Let \dsuv\ be a FS double square with $s(\dsuu)<s(\dsuv)$, then either:\\
($a$) $s(\dsuv)<e(u_{[1]})$, in which case either:\\
\indent ($a_1$) \dsuv\ is an $\alpha$-mate of \dsuu: $\longueur{v}=\longueur{u}, \longueur{V}=\longueur{U}$, or \\
\indent ($a_2$) \dsuv\ is a $\beta$-mate of \dsuu: $\longueur{v}<\longueur{u}, \longueur{V}=\longueur{U}$, or\\
\indent ($a_3$) \dsuv\ is a $\gamma$-mate of \dsuu: $\longueur{v}=\longueur{U}$, or\\
\indent ($a_4$) \dsuv\ is a $\delta$-mate of \dsuu: $\longueur{v}>\longueur{U}$;\\
\indent or \\
($b$)  $s(\dsuv)\geq e(u_{[1]})$, then either:\\
\indent ($b_1$) \dsuv\ is an $\epsilon$-mate of \dsuu, or\\
\indent ($b_2$) \dsuv\ is an $\eta$-mate of \dsuu, or\\
\indent ($b_3$) \dsuv\ is an $\zeta$-mate of \dsuu.
\end{lemma}

\begin{proof}
The proof of lemma \ref{exhaust} is provided in \cite{DFT1}, except for the ($b$) cases: the three cases that we consider are treated as just one case. Nonetheless, it is straightforward to see that the union of the $\epsilon$-mates (which have a particular definition), the $\eta$-mates (which have a particular definition) and the $\zeta$-mates (which are defined as all the non $u$-close FS double squares that are neither $\epsilon$ nor $\eta$-mates) form all of the squares that can start after $e(u_{[1]})$.
\end{proof}

\begin{definition} [family]\label{family}
Let $x$ be a word starting with a FS double square \dsuu. \\
If \dsuu\ has neither $\beta$ nor $\eta$-mates, the FS double squares that are not in the family are the $\delta$, $\epsilon$ and $\zeta$-mates of \dsuu\ and the \dsuu-family is composed of all the $\alpha$ and $\gamma$-mates of \dsuu\ that start before a FS double square that is not in the family.\\
If \dsuu\ has $\eta$-mates but no $\beta$-mates, the FS double squares that are not in the family are the $\gamma$, $\delta$ and $\zeta$-mates of \dsuu\ and the \dsuu-family is composed of the $\alpha$, $\epsilon$ and $\eta$-mates of \dsuu\ that start before a FS double square that is not in the family.\\
If \dsuu\ has $\beta$-mates, the FS double squares that are not in the family are the $\delta$ and $\zeta$-mates of \dsuu\ tand the \dsuu-family is composed of all the $\alpha, \beta, \gamma, \epsilon$ and $\eta$-mates that start before a FS double square that is not in the family.
\end{definition}


\section{New bounds}

\begin{definition}
Let $w$ be a word. The number of FS double squares in $w$ is denoted $\delta (w)$.
\end{definition}

\begin{definition}
Let \dsuu\ and \dsuv\ be two FS double square of a word $w$. The number of FS double squares of $w$ that start between $s(\dsuu)$ and $s(\dsuv)$ (excluded) is denoted as $d_w(\dsuu, \dsuv)$, or $d(\dsuu, \dsuv)$ if it is clear from context.
\end{definition}

\begin{definition} [gap, tail]
Let \dsuu\ and \dsuv\ be two FS double squares. If $ut = gv$ for two words $g$ and $t$, $T(\dsuu, \dsuv) = t$ is the tail of \dsuu\ and \dsuv, and $G(\dsuu, \dsuv) = g$ is their gap.
\end{definition}

The next theorem is the main result of this article. Its proof is intended as a blueprint for the upcoming sections.

\begin{theorem}\label{fsdsa}
There are at most $\frac{1}{2}n$ FS double squares in a word of length $n$.
\end{theorem}

\begin{proof}
Let $w$ be a word starting with a FS double square \dsuu. We provided a list of double squares relative to \dsuu: the $\alpha, \beta, \gamma, \delta, \epsilon$ and $\zeta$-mates of \dsuu\ and the different ways they can form families. We proved the exhaustivity of that list in \ref{exhaust}. Lemma \ref{induction} says that if there is another FS double square \dsuv\ in $w$, and for $w'$ the suffix of $w$ that starts with \dsuv, $\delta (w') \leq \frac{1}{2} \lvert w' \rvert - \frac{1}{2} \lvert v \rvert$ implies that $\delta (w) \leq \frac{1}{2} \lvert w \rvert - \frac{1}{2} \lvert u \rvert + d_w(\dsuu, \dsuv) - \frac{1}{2}T(\dsuu, \dsuv)$. \\
Lemmas \ref{allthefamilies} prove that, for all the possible configurations of those $\alpha, \beta\, \gamma, \delta, \epsilon, \eta$ and $\zeta$-mates of \dsuu, either $\delta (w) \leq \frac{1}{2} \lvert w \rvert - \frac{1}{2} \lvert u \rvert$ or there is a FS double square of $w$, \dsuv,  that verifies $d_w(\dsuu, \dsuv) - \frac{1}{2}T(\dsuu, \dsuv) \leq 0$. Because of the factorization of FS double squares, their minimal length is $10$, and $w'$, the suffix of a word $w$ that starts with the rightmost FS double square \dsuv\ of $w$ has $\delta (w') \leq \frac{1}{2} (\lvert w' \rvert - \longueur{v})$.
 We then use backward induction, by applying the hypothesis to the starting point of the different suffixes of $w$ that start with a FS double squares. 
 \end{proof}

\begin{corollary}
There are less than $\frac{3}{2}n$ distinct squares in a word of length $n$.
\end{corollary}

\begin{proof}
If $\lvert w \rvert = n$, by \ref{fsdsa}, there are at most $\frac{n}{2}$ distinct double squares in $w$. If at each position where a FS double square does not start, a square does, $w$ has less than $\frac{3}{2}n$ distinct squares.
\end{proof}


\section{Induction}\label{section5}

The induction relies on the number of FS double squares that start between the starting points of two given FS double squares \dsuu\ and \dsuv , $d(\dsuu, \dsuv)$.

\begin{lemma}\label{induction} [Induction hypothesis]
Let $w$ be a word starting with a FS double square \dsuu. Let \dsuv\ be another FS double square of $w$ that ends after position $\longueur{u}+1$, and $w'$ be the suffix of $w$ that starts with \dsuv. Then $\delta (w') \leq \frac{1}{2} \lvert w' \rvert - \frac{1}{2} \lvert v \rvert$ implies that $\delta (w) \leq \frac{1}{2} \lvert w \rvert - \frac{1}{2} \lvert u \rvert + d(\dsuu, \dsuv) - \frac{1}{2}T(\dsuu, \dsuv)$.
\end{lemma}

\begin{proof}
For convenience, write $ut = gv$ ($t=T(\dsuu, \dsuv), g=G(\dsuu, \dsuv)$).
Suppose that $\delta (w') \leq \frac{1}{2} \lvert w' \rvert - \frac{1}{2} \lvert v \rvert$. Because $\delta(w) = \delta(w') + d(\dsuu, \dsuv)$ , $\delta(w) \leq  \frac{1}{2} \lvert w' \rvert - \frac{1}{2} \lvert v \rvert +d(\dsuu, \dsuv) = \frac{1}{2}((\longueur{gw'} -\longueur{g}) - (\longueur{ut} - g))+d(\dsuu, \dsuv)= \frac{1}{2}(\longueur{w}-\longueur{u} -\longueur{t})+d(\dsuu, \dsuv)$.
\end{proof}

We need to show that either there is a FS \dsuv\ (that we can choose) such that $d(\dsuu, \dsuv) - \frac{1}{2}T(\dsuu, \dsuv) \leq 0$ or, that $\delta (w) \leq \frac{1}{2} \lvert w \rvert - \frac{1}{2} \lvert u \rvert$ if there is none. As in \cite{DFT1}, the case where there is none corresponds to the case where all the double squares of $w$ are in the \dsuu-family: we will prove that if $w$ starts with \dsuu\ then $\sigma(w) \leq \frac{\lvert w \rvert - \lvert u \rvert}{2}$.
We will then prove that if there is at least one FS double square that is not in the \dsuu-family, the first one of them, \dsuv, verifies $d(\dsuu, \dsuv) - \frac{1}{2}T(\dsuu, \dsuv) \leq 0$.


\section{Handling the different families}

\begin{lemma}\label{allthefamilies}
Let $w$ be a word starting with a FS double square \dsuu. For all the possible configurations of $\alpha, \beta\, \gamma, \delta, \epsilon, \zeta$ and $\eta$-mates of \dsuu, either $\delta (w) \leq \frac{1}{2} \lvert w \rvert - \frac{1}{2} \lvert u \rvert$ or there is a FS double square of $w$, \dsuv,  that verifies $d(\dsuu, \dsuv) - \frac{1}{2}T(\dsuu, \dsuv) \leq 0$
\end{lemma}

\begin{proof}
By definition of the families, if \dsuu\ has no $\beta$ nor $\eta$-mates, then the \dsuu-family is composed only of the $\alpha$ and $\gamma$-mates of \dsuu.
\begin{itemize}
\item The case where \dsuu\ has only $\alpha$-mates in its family is treated in lemma \ref{alphaonly} if there are no other FS double squares in $w$ and in lemma \ref{alphaothers} if there is any.
\item The case where \dsuu\ has only $\alpha$ and $\gamma$-mates in its family (or only $\gamma$-mates) is treated in lemma \ref{gammaonly} if there are no other FS double squares in $w$ and in lemma \ref{gammaothers} if there is any: albeit those lemmas are presented for families composed of $\alpha, \beta$ and $\gamma$-mates, their proofs rely on the properties induced by the presence of $\gamma$-mates, and as such, encompass the cases where \dsuu\ only has $\alpha$ and $\gamma$-mates in its family.
\end{itemize}
By definition of the families, if \dsuu\ has no $\beta$ but has $\eta$-mates, then the \dsuu-family is composed of the $\alpha$, $\epsilon$ and $\eta$-mates of \dsuu\ (the case with only $\alpha$ and $\epsilon$-mates is treated in lemma \ref{alphaothers}). By property \ref{gammaeta}, there cannot be $\gamma$ and $\eta$-mates at the same time. 

\begin{itemize}
\item The case where \dsuu\ has only $\alpha$ and $\eta$-mates in its family (or only $\eta$-mates) is treated in lemma \ref{alphaetas} if there are no other FS double squares in $w$ and in lemma \ref{alphaetasothers} if there is any.
\item The case where \dsuu\ has only $\alpha$, $\eta$ and $\epsilon$-mates in its family (or only $\eta$ and $\epsilon$-mates, or only $\epsilon$-mates) is treated in lemma \ref{alphaetaepsilonlemma} if there are no other FS double squares in $w$ and in lemma \ref{alphaetaepsilonlemmaothers} if there is any.
\end{itemize}

If \dsuu\ has $\beta$-mates, then all the $\alpha, \beta, \gamma, \epsilon$ and $\eta$-mates are in the family. By property \ref{gammas2}, there cannot be $\gamma$ and $\epsilon$-mates at the same time. By property \ref{gammaeta}, there cannot be $\gamma$ and $\eta$-mates at the same time 
\begin{itemize}

\item The case where \dsuu\ has only $\alpha$ and $\beta$-mates in its family (or only $\beta$-mates) is treated in lemma \ref{betaonly} if there are no other FS double squares in $w$ and in lemma \ref{betaothers} if there is any.
\item The case where \dsuu\ has only $\alpha, \beta$ and $\gamma$-mates in its family (or only $\beta$ and $\gamma$-mates) is treated in lemma \ref{gammaonly} if there are no other FS double squares in $w$ and in lemma \ref{gammaothers} if there is any.
\item The case where \dsuu\ has only $\alpha, \beta$ and $\epsilon$-mates in its family (or only $\beta$ and $\epsilon$-mates) is treated in lemma \ref{epsilonsall} for both the cases where there are no other FS double squares in $w$ and where there is any.

\item The case where \dsuu\ has only $\alpha, \beta$ and $\eta$-mates or only $\alpha, \beta, \epsilon$ and $\eta$-mates in its family (or only $\beta$ and $\eta$-mates, or only $\beta$ and $\epsilon$-mates, or only $\beta, \eta$ and $\epsilon$-mates) is treated in lemma \ref{epsilonetasonly} if there are no other FS double squares in $w$ and in lemma \ref{epsilonetaothers} if there is any.

\end{itemize}
\end{proof}

\begin{lemma}\label{alphaonly}[$\alpha$ only]
Let $w$ be a word starting with a FS double square \dsuu\ that has only $\alpha$-mates in its family. If all the squares of $w$ are \dsuu\ and its family, then $\delta (w) \leq \frac{1}{2} \lvert w \rvert - \frac{1}{2} \lvert u \rvert$.
\end{lemma}

\begin{proof}
By property \ref{alphas}, there are at most $\lvert \lcp(\prefix{u}\suffix{u}, \suffix{u}\prefix{u}) \rvert$ $\alpha$-mates, and by lemma \ref{dft}, $\lvert \lcp(\prefix{u}\suffix{u}, \suffix{u}\prefix{u}) \rvert \leq  \lvert \generateur{u} \rvert$. Hence, there are at most $ \lvert \generateur{u} \rvert$ FS double squares in the \dsuu-family. Now, by lemma \ref{lams}, $\lvert w \rvert \geq 2 \longueur{U} =  2(e_1 + e_2)\lvert \generateur{u} \rvert + 2\lvert \prefix{u} \rvert$ and $\frac{1}{2} \lvert w \rvert - \frac{1}{2} \lvert u \rvert \geq (e_1 + e_2 - \frac{e_1}{2})\lvert \generateur{u} \rvert + \frac{1}{2} \lvert \prefix{u} \rvert > \lvert \generateur{u} \rvert$
\end{proof}

\begin{lemma}\label{alphaothers}[$\alpha$ and non $u$-family]
Let $w$ be a word starting with a FS double square \dsuu\ that has only $\alpha$-mates in its family. If there are FS double squares in $w$ other than \dsuu\ and its family, then there exists a FS double square of $w$, \dsuv , that verifies $d(\dsuu, \dsuv) - \frac{1}{2}T(\dsuu, \dsuv) \leq 0$.
\end{lemma}

\begin{proof}
Let \dsuv\ be the first (leftmost) FS double square that is not in the \dsuu-family. We need to consider 3 cases: the case where \dsuv\ is a $\delta$-mate of \dsuu , where \dsuv\ is an $\epsilon$-mate of \dsuu\ and where \dsuv\ is a  $\zeta$-mate of \dsuu. In those four cases, the size of the \dsuu-family,  hence the number of FS double squares that start between \dsuu\ and \dsuv, $d(\dsuu, \dsuv)$, is at most $\lvert p \rvert < \lvert \generateur{u} \rvert$ by lemma \ref{alphaonly}. \\
\begin{itemize}
\item If \dsuv\ is a $\delta$-mate, by Property \ref{deltas}, $T(\dsuu, \dsuv) \geq \lvert U \rvert \geq 2 \longueur{p}$.

\item If \dsuv\ is an $\epsilon$-mate of \dsuu, 
by property \ref{epsilons}, ${T(\mathcal{U}, \dsuv)} > \lvert U \rvert +\longueur{p} - 2 \longueur{\generateur{u}} $.\\
If $e_1\geq2$, $\lvert U \rvert +\longueur{p} - 2 \longueur{\generateur{u}} \geq \longueur{\generateur{u}}+\longueur{p}$ and $T(\dsuu, \dsuv) \geq 2 \lvert p \rvert$. \\
If $e_1=1$ then $w$ has the prefix: 
\[ w'=\generateur{u}\prefix{u}\generateur{u}\generateur{u}\prefix{u}\generateur{u}r.\]
For $u^2$ not to be repeated, $\longueur{lcp(\generateur{u},r)}<\longueur{\prefix{u}}$. Following the proof of property \ref{alphas}, the number of $\alpha$-mates of \dsuu\ is also bounded by $\longueur{lcp(\generateur{u},r)}$, hence is bounded by $\min(\longueur{lcp(\generateur{u},r)}, \longueur{p})\leq\min(\longueur{\prefix{u}}, \longueur{p})$, while $ \lvert U \rvert +\longueur{p} - 2 \longueur{\generateur{u}} = \longueur{\prefix{u}}+\longueur{p}$.

\item If \dsuv\ is a $\zeta$-mate of \dsuu, 
by property \ref{zetas1}, $T(\dsuu, \dsuv) \geq \lvert U \rvert - \lvert \generateur{u} \rvert + \lvert p \rvert = (e_1+e_2-1)\lvert \generateur{u} \rvert + \lvert \prefix{u} \rvert + \lvert p \rvert \geq 2 \lvert p \rvert$ since $\longueur{p} < \longueur{\generateur{u}}$ by lemma \ref{dft}. 
\end{itemize} 

\end{proof}

\begin{lemma}\label{alphaetas}[$\alpha$ and $\eta$ only]
Let $w$ be a word starting with a FS double square \dsuu\ that has only $\alpha$ and $\eta$-mates in its family. If all the squares of $w$ are \dsuu\ and its family, then $\delta (w) \leq \frac{1}{2} \lvert w \rvert - \frac{1}{2} \lvert u \rvert$.
\end{lemma}

\begin{proof}
By property \ref{alphas}, there are at most $\lvert \lcp(\prefix{u}\suffix{u}, \suffix{u}\prefix{u}) \rvert$ $\alpha$-mates, and by property \ref{dft}, $\lvert \lcp(\prefix{u}\suffix{u}, \suffix{u}\prefix{u}) \rvert + \lvert \lcs(\prefix{u}\suffix{u}, \suffix{u}\prefix{u}) \rvert = \longueur{p} + \longueur{s} \leq  \lvert \generateur{u} \rvert$. \\
By property \ref{etassimple}, \dsuu\ has at most $\longueur{p}+\longueur{s}-\longueur{\generateur{u}}+\longueur{\prefix{u}}$ $\eta$-mates and $\longueur{w} \geq \longueur{Uu} + 3\longueur{\generateur{u}}+\longueur{\prefix{u}}$.\\
It follows that $\delta (w) \leq \longueur{p}+\longueur{\prefix{u}} \leq 2\longueur{\generateur{u}}$ while $\frac{1}{2} \lvert w \rvert - \frac{1}{2} \lvert u \rvert \geq  \frac{5}{2} \longueur{\generateur{u}}$.

\end{proof}

\begin{lemma}\label{alphaetasothers}[$\alpha$, $\eta$ and a non $u$-family]
Let $w$ be a word starting with a FS double square \dsuu\ that has only $\alpha$ and $\eta$-mates in its family. If there are FS double squares in $w$ other than \dsuu\ and its family, then there exists a FS double square of $w$, \dsuv , that verifies $d(\dsuu, \dsuv) - \frac{1}{2}T(\dsuu, \dsuv) \leq 0$.
\end{lemma}

\begin{proof}
By lemma \ref{alphaetas}, the size of the family is less than $\longueur{\generateur{u}}+\longueur{\prefix{u}}$. By property \ref{gammaeta}, there cannot be $\gamma$ and $\eta$-mates at the same time\\
Let \dsuv\ be the first (leftmost) FS double square that is not in the \dsuu-family. We need to consider 2 cases: the case where \dsuv\ is a $\delta$-mate of \dsuu, and the case where \dsuv\ is a  $\zeta$-mate of \dsuu. 
\begin{itemize}
\item If \dsuv\ is a $\delta$-mate of \dsuu, then by property \ref{deltaeta}, $T(\dsuu, \dsuv) \geq (e_1+e_2+2)\longueur{\generateur{u}} + 2 \longueur{\prefix{u}}$.

\item If \dsuv\ is a $\zeta$-mate of \dsuu, then by property \ref{zetaeta}, $T(\dsuu, \dsuw)\geq (e_1+e_2+n-1)\longueur{\generateur{u}}+2\longueur{\prefix{u}} - \longueur{s}\geq (e_1+e_2)\longueur{\generateur{u}} + 2 \longueur{\prefix{u}}$.

\end{itemize}
\end{proof}

\begin{lemma}[$\alpha$, $\eta$ and $\epsilon$ only]\label{alphaetaepsilonlemma}
Let $w$ be a word starting with a FS double square \dsuu\ that has only $\alpha$, $\epsilon$ and $\eta$-mates in its family. If all the squares of $w$ are \dsuu\ and its family, then $\delta (w) \leq \frac{1}{2} \lvert w \rvert - \frac{1}{2} \lvert u \rvert$.
\end{lemma}

\begin{proof}
By property \ref{alphas}, there are at most $\longueur{p}$ $\alpha$-mates.\\
By property \ref{etassimple}, there are at most $\longueur{p}+\longueur{s}-\longueur{\generateur{u}}+\longueur{\prefix{u}}$ $\eta$-mates.\\
Because \dsuu\ has an $\epsilon$-mate \dsuv\ and an $\eta$-mate \dsuw, by property \ref{epsilonetabis}, \dsuu\ has at most $2\longueur{\generateur{u}}-\longueur{s}-\longueur{p}$ $\epsilon$-mates. \\
Hence the size of the family is at most $\longueur{\generateur{u}}+\longueur{p}+\longueur{\prefix{u}}\leq 2 \longueur{\generateur{u}}+\longueur{\prefix{u}}$.\\
Because \dsuu\ has an $\eta$-mate, by property \ref{etas2}, \[w= \generateur{u}^{e_1}\prefix{u}\generateur{u}^{e_1+e_2}\prefix{u}\generateur{u}^{2n-1}\prefix{u}r'\]
 for an integer $n\geq2$ and a suffix $r'$ of $w$. Also, by corollary \ref{etas5}, \dsuu's first exponent is strictly greater than 2 and $\frac{1}{2} \lvert w \rvert - \frac{1}{2} \lvert u \rvert \geq (\frac{e_1+e_2}{2}+n-\frac{1}{2})\longueur{\generateur{u}}+\longueur{\prefix{u}}\geq 2 \longueur{\generateur{u}}+\longueur{\prefix{u}}$.
\end{proof}

\begin{lemma}[$\alpha$, $\eta$, $\epsilon$ and a non-$u$-familly]\label{alphaetaepsilonlemmaothers}
Let $x$ be a word starting with a FS double square \dsuu\ that has only $\alpha$, $\eta$ and $\epsilon$-mates in its family. If there are FS double squares in $w$ other than \dsuu\ and its family, then there exists a FS double square of $x$, \dsuv , that verifies $d(\dsuu, \dsuv) - \frac{1}{2}T(\dsuu, \dsuv) \leq 0$.
\end{lemma}

\begin{proof}
Because \dsuu\ has an $\eta$-mate \dsuw, $\longueur{w} = n \longueur{\generateur{u}}$ for some $n$ (remark that $n$ is uniquely defined by corollary \ref{etas3}). By lemma \ref{alphaetaepsilonlemma}, the size of the family is at most $2\longueur{\generateur{u}}+\longueur{\prefix{u}}$. By corollary \ref{etas5}, \dsuu's first exponent is strictly greater than 2, and by property \ref{etas2}, $n\geq2$.\\
Let \dsuv\ be the first (leftmost) FS double square that is not in the \dsuu-family. By property \ref{gammaeta}, there cannot be $\gamma$ and $\eta$-mates at the same time. We need to consider 2 cases: the case where \dsuv\ is a $\delta$-mate of \dsuu,  where \dsuv\ is a  $\zeta$-mate of \dsuu. 
\begin{itemize}
\item If \dsuv\ is a $\delta$-mate of \dsuu, then by property \ref{deltaeta}, $T(\dsuu, \dsuv) \geq (e_1+e_2+2)\longueur{\generateur{u}} + 2 \longueur{\prefix{u}}$.
\item If \dsuv\ is a $\zeta$-mate of \dsuu, then by property \ref{zetaeta}, $T(\dsuu, \dsuv)\geq (e_1+e_2+n-1)\longueur{\generateur{u}}+2\longueur{\prefix{u}}$ and because \dsuu\ has an $\eta$-mate, by corollary \ref{etas5}, \dsuu's first exponent is strictly greater than 2.
\end{itemize}

\end{proof}

\begin{lemma}\label{betaonly}[$\alpha$ and $\beta$ only]
Let $w$ be a word starting with a FS double square \dsuu\ that has only $\alpha$ and $\beta$-mates in its family. If all the squares of $w$ are \dsuu\ and its family, then $\delta (w) \leq \frac{1}{2} \lvert w \rvert - \frac{1}{2} \lvert u \rvert$.
\end{lemma}

\begin{proof}
By property \ref{alphas}, there are at most $\lvert \lcp(\prefix{u}\suffix{u}, \suffix{u}\prefix{u}) \rvert$ $\alpha$-mates, and by property \ref{betas}, $\lfloor \frac{e_1-e_2}{2} \rfloor (\lvert \lcp(\prefix{u}\suffix{u}, \suffix{u}\prefix{u}) \rvert + \lvert \lcs(\prefix{u}\suffix{u}, \suffix{u}\prefix{u}) \rvert+1)$ $\beta$-mates. The size of $w$ is  $\lvert w \rvert \geq 2(e_1 + e_2)\lvert \generateur{u} \rvert + 2\lvert \prefix{u} \rvert$, and $\frac{1}{2} \lvert w \rvert - \frac{1}{2} \lvert u \rvert \geq (\frac{e_1}{2} + e_2)\lvert \generateur{u} \rvert  > (\frac{e_1 - e_2}{2} +1)\lvert \generateur{u} \rvert$ since $e_2 \geq 1$.
\end{proof}

\begin{lemma}\label{betaothers}[$\alpha, \beta$ and non-$u$-family]
Let $w$ be a word starting with a FS double square \dsuu\ that has only $\alpha$ and $\beta$-mates in its family. If there are FS double squares in $w$ other than \dsuu\ and its family, then there exists a FS double square of $w$, \dsuv , that verifies $d(\dsuu, \dsuv) - \frac{1}{2}T(\dsuu, \dsuv) \leq 0$.
\end{lemma}

\begin{proof}
Let \dsuv\ be the first (leftmost) FS double square that is not in the \dsuu-family, \dsuv\ is either a $\delta$ or a $\zeta$-mate of \dsuu. By property \ref{alphas}, \dsuu\ has at most $\longueur{p}$ $\alpha$-mates and by property \ref{betas}, there are at most $\lfloor \frac{e_1-e_2}{2} \rfloor(\longueur{p}+\longueur{s}+1)$ $\beta$-mates. Summing the $\alpha$ and $\beta$-mates gives $d(\dsuu, \dsuv) \leq \lfloor \frac{e_1-e_2}{2} \rfloor \longueur{\generateur{u}} + \longueur{p}$.
\begin{itemize}
\item If \dsuv\ is a $\delta$-mate, by Property \ref{deltas}, $T(\dsuu, \dsuv) \geq \lvert U \rvert$ and $\frac{1}{2}T(\dsuu, \dsuv) \geq d(\dsuu, \dsuv)$.
\item If \dsuv\ is a $\zeta$-mate, by Property \ref{zetas1}, $T(\dsuu, \dsuv) \geq \lvert U \rvert - \lvert \generateur{u} \rvert + \lvert p \rvert = (e_1-1)\longueur{\generateur{u}} + (e_2\longueur{\generateur{u}} + \longueur{p}) \geq 2 d(\dsuu, \dsuv)$.
\end{itemize}
\end{proof}

\begin{lemma}\label{gammaonly}[$\alpha$, $\beta$ and $\gamma$ only]
Let $w$ be a word starting with a FS double square \dsuu\ that has only $\alpha$, $\beta$ and $\gamma$-mates in its family. If all the squares of $w$ are \dsuu\ and its family, then $\delta (w) \leq \frac{1}{2} \lvert w \rvert - \frac{1}{2} \lvert u \rvert$.
\end{lemma}

\begin{proof}
Suppose that the last canonical $\gamma$-mate starts at position $t\longueur{\generateur{u}}+1$, hence the \dsuu-family has at most $(t+1)\longueur{\generateur{u}}$ elements. Following the proof of \ref{gamas}, we get that $\longueur{w} \geq t\longueur{\generateur{u}}+ 2((2(e_1+e_2)-k)\longueur{\generateur{u}} + \longueur{\prefix{u}})$ and  $\frac{1}{2} \lvert w \rvert - \frac{1}{2} \lvert u \rvert = \frac{t+3e_1+4e_2-2k}{2}\longueur{\generateur{u}} > (t+1)\longueur{\generateur{u}}$ as $e_1-k > t$.
\end{proof}

\begin{lemma}\label{gammaothers}[$\alpha, \beta$, $\gamma$ and non-$u$-family.]
Let $w$ be a word starting with a FS double square \dsuu\ that has only $\alpha$, $\beta$ and $\gamma$-mates in its family. If there are FS double squares in $w$ other than \dsuu\ and its family, then there exists a FS double square of $w$, \dsuv , that verifies $d(\dsuu, \dsuv) - \frac{1}{2}T(\dsuu, \dsuv) \leq 0$.
\end{lemma}

\begin{proof}
Let \dsuv\ be the first (leftmost) FS double square that is not in the \dsuu-family, \dsuv\ is either a $\delta$ or a $\zeta$-mate of \dsuu. Suppose that the last canonical $\gamma$-mate starts at position $t\longueur{\generateur{u}}+1$, hence the \dsuu-family has at most $(t+1)\longueur{\generateur{u}}$ elements. .
\begin{itemize}
\item If \dsuv\ is a $\delta$-mate, by Property \ref{deltagama}, $\frac{1}{2}T(\dsuu, \dsuv) \geq d(\dsuu, \dsuv)$.
\item If \dsuv\ is a $\zeta$-mate, by Property \ref{zetagama}, $\frac{1}{2}T(\dsuu, \dsuv) \geq d(\dsuu, \dsuv)$.
\end{itemize}
\end{proof}

\begin{lemma}\label{epsilonsall}[$\alpha, \beta$, $\epsilon$-mates both with and without a non-$u$-family.]
If $w$ is a word starting with a FS double square \dsuu\ that has at least one $\epsilon$-mate in its family, then there is a FS double square of $w$, \dsuv,  that verifies $d(\dsuu, \dsuv) - \frac{1}{2}T(\dsuu, \dsuv) \leq 0$.
\end{lemma}

\begin{proof}
By property \ref{gammas2}, if \dsuu\ has $\epsilon$-mates in its family, then it has no $\gamma$-mates. By property \ref{alphas} there are at most $\longueur{p}$ $\alpha$-mates. By property \ref{betas}, if $tr(\dsuu, w) > \lfloor \frac{e_1-e_2}{2}\rfloor$ there are at most $n = \lfloor \frac{e_1-e_2}{2}\rfloor(\lvert p \rvert + \lvert s \rvert +1)$ $\beta$-mates and $n = ((tr(\dsuu, w)-1)(\lvert p \rvert + \lvert s \rvert +1) + \min (\lvert \lcp(\generateur{u}, r)\rvert, \lvert p \rvert) + \lvert s \rvert +1)^+$ if $tr(\dsuu, w) \leq \lfloor \frac{e_1-e_2}{2}\rfloor$.
Let \dsuv\ be the first $\epsilon$-mate of \dsuu, thus:
 \[ d(\dsuu, \dsuv) \leq \lfloor \frac{e_1-e_2}{2}\rfloor \longueur{\generateur{u}} + \longueur{p}\] 
 if $tr(\dsuu, w) > \lfloor \frac{e_1-e_2}{2}\rfloor$ and if $tr(\dsuu, w) \leq \lfloor \frac{e_1-e_2}{2}\rfloor$:
 \[ d(\dsuu, \dsuv) \leq tr(\dsuu, w) \longueur{\generateur{u}} + \min (\longueur{\lcp(\generateur{u}, r)}, \longueur{p}).\] 
 In order not to be repeated, $v$ has to start at $s(v) \geq \longueur{Uu} + 1 + \lvert p \rvert - 3 \lvert \generateur{u} \rvert$ since $\longueur{v} \leq \longueur{\generateur{u}}$ and $v^2$ has to end after $\longueur{Uu} + 1 + \lvert p \rvert - \lvert \generateur{u} \rvert $, hence $v$ must end at $e(v) \geq \frac{\longueur{Uu} + 1 + \lvert p \rvert - \lvert \generateur{u} \rvert + s(v)}{2}$, which is increasing with respect to $s(v)$. Therefore \begin{align*}T(\dsuu, \dsuv) &\geq (e_1+e_2-2)\lvert \generateur{u} \rvert + \lvert \prefix{u} \rvert + \lvert p \rvert \\ &\geq (e_1-e_2 +2e_2 -2)\lvert \generateur{u} \rvert + \lvert \prefix{u} \rvert + \lvert p \rvert .\end{align*} Set $q(\dsuu, \dsuv) = d(\dsuu, \dsuv) - \frac{1}{2}T(\dsuu, \dsuv)$.
\begin{itemize}
\item If $e_2 \geq 2$, then $d(\dsuu, \dsuv) - \frac{1}{2}T(\dsuu, \dsuv) \leq 0$. \\

\item If $e_2 = 1$, and $tr(\dsuu, w) < e_1-e_2-\lfloor \frac{e_1-e_2}{2} \rfloor $, either $e_1 \equiv 0 \pmod 2$ hence $\lfloor \frac{e_1-e_2}{2} \rfloor - \frac{e_1-e_2}{2} = -\frac{1}{2}$ and one needs to notice that $\longueur{p} < \longueur{\generateur{u}}$ to conclude that $q(\dsuu, \dsuv)$ is negative; or $e_1 \equiv 1 \pmod 2$ hence $tr(\dsuu, w) \leq (\lfloor \frac{e_1-e_2}{2} \rfloor -1)$, $d(\dsuu, \dsuv) \leq (\lfloor \frac{e_1-e_2}{2} \rfloor -1)\longueur{\generateur{u}} + \min (\longueur{\lcp(\generateur{u}, r)}, \longueur{p})$ and $q(\dsuu, \dsuv)$ is negative.

\item If $e_2 = 1$, and $tr(\dsuu, w)) = e_1-e_2-\lfloor \frac{e_1-e_2}{2} \rfloor $, either $e_1 \equiv 1 \pmod 2$ and, by property \ref{betas2}, $\longueur{\lcp (\generateur{u}, r)} \leq \longueur{\prefix{u}}$, that is $\longueur{\prefix{u}} + \longueur{p} \geq 2 \longueur{ \min (\lvert \lcp(\generateur{u}, r)\rvert, \lvert p \rvert)}$ and $d(\dsuu, \dsuv) - \frac{1}{2}T(\dsuu, \dsuv) \leq 0$; or $e_1 \equiv 0 \pmod 2$ hence $\lfloor \frac{e_1-e_2}{2}\rfloor = \frac{e_1-e_2-1}{2}$, and $d(\dsuu, \dsuv) \leq \frac{e_1-e_2-1}{2}\rfloor \longueur{\generateur{u}} + \longueur{p}$ and since $\longueur{p} < \longueur{\generateur{u}}$, $q(\dsuu, \dsuv)$ is negative.

\end{itemize}

\end{proof}

\begin{lemma}\label{epsilonetasonly}[$\alpha$, $\beta$, $\epsilon$ and $\eta$ only]
Let $w$ be a word starting with a FS double square \dsuu\ that has only $\alpha$, $\beta$, $\epsilon$ and $\eta$-mates in its family. If all the squares of $w$ are \dsuu\ and its family, then $\delta (w) \leq \frac{1}{2} \lvert w \rvert - \frac{1}{2} \lvert u \rvert$.
\end{lemma}

\begin{proof}
By property \ref{alphas}, \dsuu\ has at most $\lvert p \rvert$ $\alpha$-mates.\\
By property \ref{betas}, \dsuu\ has at most $\lfloor \frac{e_1-e_2}{2}\rfloor(\lvert p \rvert + \lvert s \rvert+1)$ $\beta$-mates.\\
By property \ref{epsilonetabis}, \dsuu\ has at most $2\longueur{\generateur{u}}-\longueur{p}-\longueur{s}$ $\epsilon$-mates.\\
By property \ref{etassimple}, \dsuu\ has at most $\longueur{p}+\longueur{s}-\longueur{\generateur{u}}+\longueur{\prefix{u}}$ $\eta$-mates.\\
Summing up all of the above, we get that the size of the \dsuu-family is bounded by $(\lfloor \frac{e_1-e_2}{2}\rfloor+1)\longueur{\generateur{u}}+\longueur{\prefix{u}}+\longueur{p} \leq (\lfloor \frac{e_1-e_2}{2}\rfloor+2)\longueur{\generateur{u}}+\longueur{\prefix{u}}-\longueur{s}$.\\
Since \dsuu\ has an $\eta$-mate, by property \ref{etas2}, $w$ has the prefix $\generateur{u}^{e_1}\prefix{u}\generateur{u}^{e_1+e_2}\prefix{u}\generateur{u}^{2n-1}\prefix{u}$, with $n\geq2$, $u=\generateur{u}^{e_1}\prefix{u}$, thus $\frac{1}{2} \lvert w \rvert - \frac{1}{2} \lvert u \rvert \geq (\frac{e_1-e_2-1}{2}+e_2+n)\longueur{\generateur{u}}+\longueur{\prefix{u}}$.
\end{proof}

\begin{lemma}\label{epsilonetaothers}[$\alpha$, $\beta$, $\epsilon$ and $\eta$ and a non \dsuu-family]
Let $w$ be a word starting with a FS double square \dsuu\ that has only $\alpha$, $\beta$, $\eta$ and $\epsilon$-mates in its family. If there are FS double squares in $w$ other than \dsuu\ and its family, then there exists a FS double square of $w$, \dsuv , that verifies $d(\dsuu, \dsuv) - \frac{1}{2}T(\dsuu, \dsuv) \leq 0$.
\end{lemma}

\begin{proof}
By lemma \ref{betaetatous}, the size of the \dsuu-family is bounded by $(\frac{e_1-e_2-1}{2}+2)\longueur{\generateur{u}}+\longueur{\prefix{u}}-\longueur{s}$.\\
Let \dsuv\ be the first (leftmost) FS double square that is not in the \dsuu-family, \dsuv\ is either a $\delta$ or a $\zeta$-mate of \dsuu. 
\begin{itemize}
\item If \dsuv\ is a $\delta$-mate of \dsuu, then by property \ref{deltaeta}, $T(\dsuu, \dsuv) \geq (e_1+e_2+2)\longueur{\generateur{u}} + 2 \longueur{\prefix{u}}$.
\item If \dsuv\ is a $\zeta$-mate of \dsuu, then by property \ref{zetaeta}, $T(\dsuu, \dsuw)\geq (e_1+e_2+n-1)\longueur{\generateur{u}}+2\longueur{\prefix{u}}-\longueur{s}$.
\end{itemize}
\end{proof}


\section{Properties}

In everything that follows, $x$ is a word starting with a FS double square $\dsuu = \UU, \generateur{u}=\prefix{u}\suffix{u}$, and we write $x = UU\generateur{u}^{tr(\dsuu, x)}r$ for a word $r$. Recall that $\lcp(\prefix{u}\suffix{u}, \suffix{u}\prefix{u})=p$, $\lcs(\prefix{u}\suffix{u}, \suffix{u}\prefix{u})=s$.

\begin{property}\label{trace}
Let $x$ be a word starting with a FS double square \dsuu, then $tr(\dsuu, x) \leq e_1-e_2$, and if there is equality, $r$ doesn't have $\prefix{u}$ has a prefix.
\end{property}

\begin{proof}
If $tr(\dsuu, x) > e_1-e_2$, $x$ would have $Uu\generateur{u}^{e_1+1}$ as a prefix, and $u^2$ would have a further occurrence in $x$ (at position $\longueur{U}+1$). In case of equality, $x$ would have $Uu\generateur{u}^{e_1}\prefix{u}$ as a prefix and the same argument applies.
\end{proof}

\begin{property}\label{alphas}
Let $x$ be a word starting with a FS double square \dsuu. Then \dsuu\ has at most $\lvert \lcp(\prefix{u}\suffix{u}, \suffix{u}\prefix{u}) \rvert = \lvert p \rvert$ $\alpha$-mates.
\end{property}

\begin{proof}
A conjugate of a square is a square, and so writing $UU\prefix{u}\suffix{u}$ will lead to $\lvert \prefix{u}\suffix{u} \rvert+1$ squares of length $\lvert U \rvert$. However, we are interested in shifts of \dsuu, not in shifts of $U$ only, thus we also need to study the possible shifts of $u$. Note that: \[UU=(\prefix{u}\suffix{u})^{e_1}\prefix{u}(\prefix{u}\suffix{u})^{e_2}(\prefix{u}\suffix{u})^{e_1}\prefix{u}(\prefix{u}\suffix{u})^{e_2}\] \[ \\ UU=(\prefix{u}\suffix{u})^{e_1}\prefix{u}(\prefix{u}\suffix{u})^{e_1}\prefix{u}(\suffix{u}\prefix{u})^{e_2}(\prefix{u}\suffix{u})^{e_2}\]\[ UU=uu(\suffix{u}\prefix{u})^{e_2}(\prefix{u}\suffix{u})^{e_2}.\]
The square $uu$ is followed by a factor $\suffix{u}\prefix{u}$. Since $uu$ has $\prefix{u}\suffix{u}$ as a prefix, the number of conjugates of $uu$ in $uu\suffix{u}\prefix{u}$ is equal to $\longueur{\lcp (\prefix{u}\suffix{u}, \suffix{u}\prefix{u})}$. It follows that the number of double square conjugates of \dsuu\ in $UU\prefix{u}\suffix{u}$ is equal to $\lvert \lcp (\prefix{u}\suffix{u}, \suffix{u}\prefix{u}) \rvert = \lvert p \rvert$.\\
The core of the interrupt insures that no other shifts of \dsuu\ are FS double squares: the small repetition of any $\alpha$-mate has to end in the core of the interrupt of \dsuu. If a FS double square \dsuv\ had its first small repetition, $v_{[1]}$, ending after the core of the interrupt, then $v_{[1]}$ would contain $w_1$, the word of length $\longueur{\generateur{u}}$ that ends with the core of the interrupt, while $v_{[3]}$ wouldn't, contradicting property \ref{ifc}.
\end{proof}

\begin{figure}[H]
\begin{align*}
\dsuu = (aa&ab) aa (aaab) . (aaab) aa (aaab) \\
\dsuv = (&aaba) aa (aaba) . (aaba) aa (aaba) \\
x = \rlap{$\overbrace{\phantom{aaabaaaaabaaabaaaaab}}^{\dsuu}$} a& \underbrace{aabaaaaabaaabaaaaaba}_{\dsuv}
\end{align*}
  \caption{\dsuv\ is an $\alpha$-mate of \dsuu.}
\end{figure}

\begin{property}\label{betas}
Let $x$ be a word starting with a FS double square \dsuu. The number of $\beta$-mates of \dsuu\ is bounded by:
\[ n = \lfloor \frac{e_1-e_2}{2}\rfloor(\lvert p \rvert + \lvert s \rvert +1) \]
if $tr(\dsuu, x) > \lfloor \frac{e_1-e_2}{2}\rfloor$ and, if $tr(\dsuu, x) \leq \lfloor \frac{e_1-e_2}{2}\rfloor$, by:
\[ n = ((tr(\dsuu, x)-1)(\lvert p \rvert + \lvert s \rvert +1) + \min (\lvert \lcp(\generateur{u}, r)\rvert, \lvert p \rvert) + \lvert s \rvert +1)^+ \]
where $e_1$ and $e_2$ are the first and second exponent, respectively, in the canonical factorization of \dsuu, and $(y)^+$ is the positive part of $y$.
\end{property}

\begin{proof}
Let $x$ be a word starting with a FS double square \dsuu.
Write $x = \generateur{u}^{e_1}\prefix{u}\generateur{u}^{e_2}\generateur{u}^{e_1}\prefix{u}\generateur{u}^{e_2}\generateur{u}^{tr(\dsuu, w)}r$ . For each integer value of $i \leq min(tr(\dsuu, w), \lfloor \frac{e_1-e_2}{2}\rfloor)$, there is a factor $\dsuv_i = \generateur{u}^{e_1-i}\prefix{u}\generateur{u}^{e_2+i}\generateur{u}^{e_1-i}\prefix{u}\generateur{u}^{e_2+i}$ starting at position $\longueur{\generateur{u}^i}+1$. We refer to those factors as shrinks of \dsuu\ (or canonical $\beta$-mates). By definition, the $\beta$-mates of \dsuu\ are the shrinks of \dsuu\ that are FS double squares and their cyclic shifts. \\
First, let's bound the number of shrinks of \dsuu\ that are FS double squares. Because those shrinks are FS double squares, they have to have their first exponent larger or equal to their second exponent. Because their exponents are $e_1-t$ and $e_2+t$, and because the first exponent of a FS double square is larger or equal to its second exponent, $t \leq\lfloor \frac{e_1-e_2}{2}\rfloor$, and there are at most $\lfloor \frac{e_1-e_2}{2}\rfloor$ of them.\\
As a consequence of property \ref{alphas}, for each integer value of $i$, $1 \leq i < \min (tr(\dsuu, x), \lfloor \frac{e_1-e_2}{2}\rfloor)$, the FS double square $\dsuv_i = \generateur{u}^{e_1-i}\prefix{u}\generateur{u}^{e_2+i}\generateur{u}^{e_1-i}\prefix{u}\generateur{u}^{e_2+i}$ has $\longueur{s}$ cyclic shifts on the left (cyclic shifts by a  strictly negative number of positions) and $\longueur{p}$ cyclic shifts on the right (cyclic shifts by a positive number of positions). If $tr(\dsuu, x) > \lfloor \frac{e_1-e_2}{2}\rfloor$, for $i = \frac{e_1-e_2}{2}$, the FS double square $\dsuv_i = \generateur{u}^{e_1-i}\prefix{u}\generateur{u}^{e_2+i}\generateur{u}^{e_1-i}\prefix{u}\generateur{u}^{e_2+i}$ has $\longueur{s}$ cyclic shifts on the left, and $\longueur{p}$ cyclic shifts on the right. If $tr(\dsuu, x) \leq \lfloor \frac{e_1-e_2}{2}\rfloor$, for $i = tr(\dsuu, x)$, the FS double square $\dsuv_i = \generateur{u}^{e_1-i}\prefix{u}\generateur{u}^{e_2+i}\generateur{u}^{e_1-i}\prefix{u}\generateur{u}^{e_2+i}$ has $\longueur{s}$ cyclic shifts on the left, but only $\min (\lvert \lcp(\generateur{u}, r)\rvert, \lvert p \rvert)$ cyclic shifts on the right. \\

As for the $\alpha$-mates, the core of the interrupt insures that no other conjugates of shrinks of \dsuu\ are FS double squares. If a $\beta$-mate of \dsuu\ were to end elsewhere than in the core of the interrupt of \dsuu , either $end(v_{[1]} \leq N_1$ or $end(v_{[1]} > N_1(\dsuu)+ \longueur{sp} +2$. If $end(v_{[1]} \leq N_1$, $v_{[2]}$ would contain $w_1$ (the factor of length $\longueur{\generateur{u}}$ that ends with the core of the interrupt) and $v_{[1]}$ wouldn't. If $end(v_{[1]} > N_1(\dsuu)+ \longueur{sp} +2$, $v_{[1]}$ would contain $w_2$ (the factor of length $\longueur{\generateur{u}}$ that starts with the core of the interrupt) at a certain position $i$ and $v_{[2]}$ would not. 
\end{proof}
   
\begin{figure}[H]
\begin{align*}
\dsuu = (aab) &(aab) (aab) a (aab) . (aab) (aab) (aab) a (aab) \\
\dsuv = &(aab) (aab) a (aab) (aab) . (aab) (aab) a (aab) (aab) \\ \text{and }
w = \rlap{$\overbrace{\phantom{aabaabaabaaabaabaabaabaaab}}^{\dsuu}$} aab & \underbrace{aabaabaaabaabaabaabaaabaab}_{\dsuv}
\end{align*}
  \caption{\dsuv\ is a $\beta$-mate of \dsuu.}
\end{figure}

\begin{property}\label{betas2}
Let $x$ be a word starting with a FS double square \dsuu. If \dsuu\ has $\lfloor \frac{e_1-e_2}{2} \rfloor (\longueur{p}+\longueur{s}+1)$ $\beta$-mates, then $tr(\dsuu, x)) \leq e_1-e_2-\lfloor \frac{e_1-e_2}{2} \rfloor$. Furthermore, if $tr(\dsuu, x)) = e_1-e_2-\lfloor \frac{e_1-e_2}{2} \rfloor $ then $\longueur{\lcp (\generateur{u}, r)} < \longueur{\prefix{u}}$. The same goes if \dsuu\ has $(\lfloor \frac{e_1-e_2}{2} \rfloor - 1)(\longueur{p}+\longueur{s}+1) + (\min (\lvert \lcp(\generateur{u}, r)\rvert, \lvert p \rvert) + \longueur{s}+1)$ $\beta$-mates.
\end{property}

\begin{proof}
Let \dsuv\ be the $\beta$-mate of \dsuu\ starting at position $\lvert \generateur{u}^{\lfloor \frac{e_1-e_2}{2}\rfloor}\rvert+1$.\\ 
 If $e_1+e_2 \equiv 0 \pmod 2$:
 \[ x=\generateur{u}^{e_1}\prefix{u}\generateur{u}^{e_2}\generateur{u}^{e_1}\prefix{u}\generateur{u}^{e_2}\generateur{u}^{tr(\dsuu, x)}r \]
 \[ x=\generateur{u}^{\lfloor \frac{e_1-e_2}{2} \rfloor}\generateur{u}^{\lfloor \frac{e_1+e_2}{2} \rfloor}\prefix{u}\generateur{u}^{e_2+\lfloor \frac{e_1-e_2}{2} \rfloor}\generateur{u}^{\lfloor \frac{e_1+e_2}{2} \rfloor}\prefix{u}\generateur{u}\generateur{u}^{e_2+\lfloor \frac{e_1-e_2}{2}\rfloor}\generateur{u}^{tr(\dsuu, x)-\lfloor \frac{e_1-e_2}{2}\rfloor}r,\]
 and \dsuv\ has the factorization: 
 \[\dsuv=\generateur{u}^{\lfloor \frac{e_1+e_2}{2} \rfloor}\prefix{u}\generateur{u}^{\lfloor \frac{e_1+e_2}{2} \rfloor}\generateur{u}^{\lfloor \frac{e_1+e_2}{2} \rfloor}\prefix{u}\generateur{u}^{\lfloor \frac{e_1+e_2}{2} \rfloor}.\]
  If $tr(\dsuu, x)) > \lfloor \frac{e_1-e_2}{2} \rfloor$,
 then the small repetition of \dsuv\ has another occurence at position $\longueur{U}+\lvert \generateur{u}^{\lfloor \frac{e_1+e_2}{2}\rfloor}\rvert+1$, hence $tr(\dsuu, x)) = \lfloor \frac{e_1-1}{2} \rfloor $. By the same argument, $\longueur{\lcp (\generateur{u}, r)} < \longueur{\prefix{u}}$. \\
 If $e_1+e_2 \equiv 1 \pmod 2$: 
 \[ x=\generateur{u}^{e_1}\prefix{u}\generateur{u}^{e_2}\generateur{u}^{e_1}\prefix{u}\generateur{u}^{e_2}\generateur{u}^{tr(\dsuu, x)}r \]
 \[ x=\generateur{u}^{\lfloor \frac{e_1-e_2}{2} \rfloor}\generateur{u}^{\lfloor \frac{e_1+e_2}{2} \rfloor+1}\prefix{u}\generateur{u}^{e_2+\lfloor \frac{e_1-e_2}{2} \rfloor}\generateur{u}^{\lfloor \frac{e_1+e_2}{2} \rfloor+1}\prefix{u}\generateur{u}\generateur{u}^{e_2+\lfloor \frac{e_1-e_2}{2}\rfloor}\generateur{u}^{tr(\dsuu, x)-\lfloor \frac{e_1-e_2}{2}\rfloor}r,\]
and \dsuv\ has the factorization: 
 \[\dsuv=\generateur{u}^{\lfloor \frac{e_1+e_2}{2} \rfloor + 1}\prefix{u}\generateur{u}^{\lfloor \frac{e_1+e_2}{2} \rfloor}\generateur{u}^{\lfloor \frac{e_1+e_2}{2} \rfloor + 1}\prefix{u}\generateur{u}^{\lfloor \frac{e_1+e_2}{2} \rfloor}.\]
If $tr(\dsuu, x) > \lfloor \frac{e_1-e_2}{2} \rfloor + 1 = e_1-e_2-\lfloor \frac{e_1-e_2}{2} \rfloor$, the small repetition of \dsuv\ has another occurrence starting at position  $\longueur{U}+\lvert \generateur{u}^{\lfloor \frac{e_1+e_2}{2}\rfloor}\rvert+1$, hence $\lfloor \frac{e_1-e_2}{2} \rfloor \leq tr(\dsuu, x) \leq e_1-e_2-\lfloor \frac{e_1-e_2}{2} \rfloor$. By the same argument, if $tr(\dsuu, x) = e_1-e_2-\lfloor \frac{e_1-e_2}{2} \rfloor$, then $\longueur{\lcp (\generateur{u}, r)} < \longueur{\prefix{u}}$.\end{proof}

\begin{property}\label{gamas}
Let $x$ be a word starting with a FS double square \dsuu\ and let \dsuv\ be a $\gamma$-mate of \dsuu. Then, in the canonical factorization of \dsuv, both exponents are equal to $1$,  $\generateur{v} = \generateur{\conj{u}}^{e_1-t}\prefix{\conj{u}}\generateur{\conj{u}}^{e_2+t-k}$ and $\prefix{v}=\generateur{\conj{u}}^k$ for $\generateur{\conj{u}}$ a conjugate of $\generateur{u}$, $\prefix{\conj{u}}$ its prefix of length $\longueur{\prefix{u}}$ and the integers $k$ ($k$ is fixed), $t, e_1-t > k > e_2$ and $t>k-e_2$. It follows that $\lcp(\generateur{v}, \generateur{\conj{v}}) = \generateur{\conj{u}}^{e_1-t-k}\prefix{\conj{u}}\lcp(\prefix{u}\suffix{u}, \suffix{u}\prefix{u})$ and $\lcs(\generateur{v}, \generateur{\conj{v}}) = \lcs(\prefix{u}\suffix{u}, \suffix{u}\prefix{u})\generateur{u}^{e_2+t-k}$.
\end{property}

\begin{proof}
As for the $\beta$-mates discussed in property \ref{betas}, we will introduce canonical $\gamma$-mates. The canonical $\gamma$-mates are the ones that start at positions $t \lvert \generateur{u} \rvert + 1$, hence have $\generateur{u}$ as a prefix. Those canonical $\gamma$-mates and their cyclic shifts will form the $\gamma$-mates of \dsuu. We will then use those canonical $\gamma$-mates to assess the total number of $\gamma$-mates. \\
If $x$ has a double square $\mathcal{U}=(u, U)$ and some shifts of it as a prefix, then:
\[ x = \generateur{u}^{e_1}\prefix{u}\generateur{u}^{e_2}\generateur{u}^{e_1}\prefix{u}\generateur{u}^{e_2}\generateur{u}^{n}r, \]
for some a suffix $r$ of $x$, and an integer $n \leq e_1 - e_2$.
Now, if that double square has a canonical $\gamma$-mate $\mathcal{V}_{(t,k)} = (v_{(t,k)}, V_{(t,k)})$, starting at position $t \lvert \generateur{u} \rvert + 1$, then:
\[v_{(t,k)} = \generateur{u}^{e_1-t}\prefix{u}\generateur{u}^{e_2+t} = \underbrace{\generateur{u}^{e_1-t}\prefix{u}\generateur{u}^{e_2+t-k}}_{\generateur{{v_{(t,k)}}}}\underbrace{\generateur{u}^k}_{\prefix{{v_{(t,k)}}}}, \]
by synchronization principle, 
 $\generateur{{v_{(t,k)}}} = \generateur{u}^{e_1-t}\prefix{u}\generateur{u}^{e_2+t-k}$ and $\prefix{{v_{(t,k)}}} = \generateur{u}^{k}$ 
for  $k \leq e_2 + t$, as $\longueur{v_{(t,k)}}= \longueur{u}$ (note that this is true for any canonical $\gamma$-mate, in particular for the last one). Now,
\[ V_{(t,k)} = \generateur{u}^{e_1-t}\prefix{u}\generateur{u}^{e_2+t-k} \generateur{u}^k \generateur{u}^{e_1-t}\prefix{u}\generateur{u}^{e_2+t-k}, \]
\[ V_{(t,k)}V_{(t,k)} = \generateur{u}^{e_1-t}\prefix{u}\generateur{u}^{e_2+t-k} \generateur{u}^k \generateur{u}^{e_1-t}\prefix{u}\generateur{u}^{e_2+t-k}\generateur{u}^{e_1-t}\prefix{u}\generateur{u}^{e_2+t-k} \generateur{u}^k \generateur{u}^{e_1-t}\prefix{u}\generateur{u}^{e_2+t-k}, \]
i.e., $e_1 - t > k > e_2$, and $t > k-e_2$. It follows that $x$ has the prefix:
\[ x'' = \generateur{u}^{e_1}\prefix{u}\generateur{u}^{e_2}\generateur{u}^{e_1}\prefix{u}\generateur{u}^{e_2}\generateur{u}^{e_1-k}\prefix{u}\generateur{u}^{e_2+t-k} \generateur{u}^k \generateur{u}^{e_1-t}\prefix{u}\generateur{u}^{e_2+t-k}, \]
\[ x'' = \generateur{u}^{e_1}\prefix{u}\generateur{u}^{e_2}\generateur{u}^{e_1}\prefix{u}\generateur{u}^{e_2}\generateur{u}^{e_1-k}\prefix{u}\generateur{u}^{e_1+e_2}\prefix{u}\generateur{u}^{e_2+t-k}. \]
Thus $k$ is fixed . Now: 
\[  \prefix{{v_{(t,k)}}}\suffix{{v_{(t,k)}}} = \generateur{u}^{e_1-t}\prefix{u}\generateur{u}^{e_2+t-k}, \]
\[ \suffix{{v_{(t,k)}}}\prefix{{v_{(t,k)}}} = \generateur{u}^{e_1-t-k}\prefix{u}\generateur{u}^{e_2+t}, \]
i.e., $lcp(v_{t,k_1}v_{t,k_2}, v_{t,k_2}v_{t,k_1} ) = \generateur{u}^{e_1-t-k}\prefix{u}p$ and $lcs(v_{t,k_1}v_{t,k_2}, v_{t,k_2}v_{t,k_1} ) = s\generateur{u}^{e_2+t-k}$
for $s = lcs(\prefix{u}\suffix{u}, \suffix{u}\prefix{u})$ and $p= lcp(\prefix{u}\suffix{u}, \suffix{u}\prefix{u})$.
\end{proof}

\begin{property}\label{gammas2}
Let $x$ be a word starting with  \dsuu, a FS double square. If \dsuu\ has a $\gamma$-mate \dsuv, then \dsuu\ has no $\epsilon$-mates.
\end{property}

\begin{proof}
Since \dsuv\ is a $\gamma$-mate of \dsuu, by property \ref{gamas}, $v = \tilde{\generateur{u}}^{e_1-n}\tilde{\prefix{u}}\tilde{\generateur{u}}^{e_2+n}$, for $\tilde{\generateur{u}}$ a conjugate of $\generateur{u}$ and $\tilde{\prefix{u}}$ the prefix of length $\lvert \prefix{u} \rvert$ of $\tilde{\generateur{u}}$ and the integer $n$. 
It follows that $x$ has the prefix:
\[ x' = \generateur{u}^{e_1}\prefix{u}\generateur{u}^{e_1+e_2}\prefix{u}\generateur{u}^{e_1+e_2-k}\prefix{u}\generateur{u}^{e_1+e_2}\prefix{u}\generateur{u}^{e_2+t-k}. \]
All the double squares that could have been $\epsilon$-mates of \dsuu\ are repeated further in $x$ (namely in the factor $\generateur{u}^{e_1+e_2}\prefix{u}\generateur{u}^{e_2+t-k}$ that starts at position $\longueur{\generateur{u}^{e_1}\prefix{u}\generateur{u}^{e_1+e_2}\prefix{u}\generateur{u}^{e_1+e_2-k}\prefix{u}}+1$).
\end{proof}

\begin{property}\label{deltas}
Let $x$ be a word starting with  \dsuu, a FS double square. If \dsuu\ has a $\delta$-mate \dsuv, then $T(\dsuu, \dsuv) \geq \lvert U \rvert$.
\end{property}

\begin{proof}
Because $x$ starts with a FS double square \dsuu, $x$ has $\generateur{u}^{e_1}\prefix{u}\generateur{u}^{e_2}\generateur{u}^{e_1}\prefix{u}\generateur{u}^{e_2}r$ as a prefix. A $\delta$-mate \dsuv\ starts before $e(u_{[1]})$ and has $\lvert v \rvert > \lvert U \rvert$.\\
If $s(v_{[1]})\geq N_1(\dsuu) = \lvert u \rvert - \lcs (\prefix{u}\suffix{u}, \suffix{u}\prefix{u})$ and $T(\dsuu, \dsuv) < \lvert U \rvert$, then $v_{[1]}$ has a prefix $\conj{\generateur{u}}$ for a certain conjugate $\conj{\generateur{u}}$ of $\generateur{u}$ and $v_{[3]}$ has another conjugate of $\generateur{u}$ as a prefix, contradicting the fact that $\generateur{u}$ is primitive.\\
If $s(v_{[1]})<N_1(\dsuu)$,  $v_{[1]}$ either contains both $w_1$, the factor of length $\longueur{\generateur{u}}$ that ends with the core of the interrupt, and $w_2$, the factor of length $\longueur{\generateur{u}}$ that starts with the core of the interrupt, or just $w_2$. \\
If $v_{[1]}$ were to contain both $w_1$ and $w_2$, at positions $i$ for $w_1$ and $j$ for $w_2$ and were to end before $\longueur{Uu} + 1$, $v_{[3]}$ would have the factor $w_1$ starting at position $i' =  i + \lvert U \lvert - \lvert v \rvert \neq i$ and the factor $w_2$ starting at position $j' = j + \lvert U \lvert - \lvert v \rvert \neq j$, a contradiction.\\ 
If $w_1$ were to contain only $w_2$, at position $j$ and were to end before $\longueur{Uu} + 1$, $v_{[3]}$ would have the factor $w_2$ starting at position $j = i + \lvert U \lvert - \lvert v \rvert \neq i$, a contradiction. 

\end{proof}

\begin{lemma}\label{deltagama}
Let $x$ be a word starting with a FS double square \dsuu. If \dsuu\ has a $\gamma$-mate \dsuv\ and a $\delta$-mate \dsuw, then $d(\dsuu, \dsuw) - \frac{1}{2} T(\dsuu, \dsuw) \leq 0$.
\end{lemma}

\begin{proof}
Suppose that the last canonical $\gamma$-mate in the \dsuu-family (discussed in \ref{gamas}) starts at position $t\longueur{\generateur{u}}+1$, hence the \dsuu-family has at most $(t+1)\longueur{\generateur{u}}$ elements. Since \dsuv\ is a member of the \dsuu-family and \dsuw\ is a $\delta$-mate of \dsuu, hence not a member of that family, \dsuw\ starts later than \dsuv\ (by definition of the families, \ref{family}) and \dsuw\ must be a mate of \dsuv. By definition of the $\delta$-mates, $\longueur{w} > \longueur{U}=\longueur{v}$, hence $\longueur{w} \geq \longueur{V}$ by property \ref{exhaust}, and $T(\dsuu, \dsuw) \geq (e_1+2e_2+t-k) \longueur{\generateur{u}}$ by lemma \ref{gamas} (note that the prefix $\generateur{u}^{e_1-t}\prefix{u}$ of the $\gamma$-mate corresponds to a suffix of $u_{[1]}$ and, as such, is not part of the tail). Since $e_1-t > k$ (also by lemma \ref{gamas}),  $T(\dsuu, \dsuw) \geq 2(e_2+t) \longueur{\generateur{u}}$ while $d(\dsuu, \dsuw) \leq (t+1)\longueur{\generateur{u}}$.
\end{proof}

\begin{property}\label{epsilons}
Let $x$ be a word starting with a FS double square \dsuu. If \dsuu\ has an $\epsilon$-mate \dsuv, then $T(\dsuu, \dsuv) \geq \lvert U \rvert + \lvert p \rvert - 2\lvert \generateur{u} \rvert$.
\end{property}

\begin{proof}
Since $\lvert v \rvert \leq \lvert \generateur{u} \rvert$, if $T(\dsuu, \dsuv) < \lvert U \rvert + \lvert p \rvert - 2\lvert \generateur{u} \rvert$, then $vv$ is repeated later in $x$ and \dsuv\ is not a FS double square.
\end{proof}

\begin{property}\label{zetas1}
Let $x$ be a word starting with a FS double square \dsuu. If \dsuu\ has a $\zeta$-mate \dsuv, then $T(\dsuu, \dsuv) \geq \lvert U \rvert - \lvert \generateur{u} \rvert + \lvert p \rvert$.
\end{property}

\begin{proof}
We have $x=\generateur{u}^{e_1}\prefix{u}\generateur{u}^{e_2}\generateur{u}^{e_1}\prefix{u}\generateur{u}^{e_2}\generateur{u}^{tr(\dsuu, x)}r$.\\
Suppose in order to derive a contradiction that $T(\dsuu, \dsuv) < \lvert U \rvert - \lvert \generateur{u} \rvert + \lvert p \rvert$. 
If $s(v) \geq\longueur{UU} + tr(\dsuu, x)\lvert \generateur{u} \rvert + \lvert \lcp (\generateur{u}, r) \rvert + 1$ because a word ends after it starts, $T(\dsuu, \dsuv) \geq \lvert U \rvert$.\\
If $s(v) < \longueur{UU} + tr(\dsuu, x)\lvert \generateur{u} \rvert + \lvert \lcp (\generateur{u}, r) \rvert + 1$, 
by definition of the $\zeta$-mates, $\forall n \geq 2, v \neq \tilde{\generateur{u}}^n$, that is $v = \conj{\generateur{u}}^{f_1}\conj{\prefix{u}}$ for $\conj{\generateur{u}}$ a conjugate of $\generateur{u},\ \conj{\prefix{u}}$ a proper prefix of $\conj{\generateur{u}}$ (not necessarily of size $\longueur{\prefix{u}}$) and an integer $f_1, 1 \leq f_1\leq e_1$ and there is a factor $\conj{\generateur{u}}$ at position $end(v)+1$, contradicting the synchronization principle, \ref{sp}.
\end{proof}

\begin{property}\label{zetagama}
Let $x$ be a word starting with a FS double square \dsuu. If \dsuu\ has a $\gamma$-mate \dsuv\ and a $\zeta$-mate \dsuw, then $d(\dsuu, \dsuw) - \frac{1}{2} T(\dsuu, \dsuw) \leq 0$.
\end{property}

\begin{proof}
Suppose that the last canonical $\gamma$-mate (discussed in \ref{gamas}) starts at position $t\longueur{\generateur{u}}+1$, hence \dsuu-family has at most $(t+1)\longueur{\generateur{u}}$ elements. Because \dsuu\ has a $\gamma$-mate, $x$ has the prefix:
\[ x'' = \generateur{u}^{e_1}\prefix{u}\generateur{u}^{e_2}\generateur{u}^{e_1}\prefix{u}\generateur{u}^{e_2}\generateur{u}^{e_1-k}\prefix{u}\generateur{u}^{e_1+e_2}\prefix{u}\generateur{u}^{e_2+t-k}, \]
with $k>e_2$, hence $tr(\dsuu, x) = e_1-k>t$, by lemma \ref{gamas}. \\
Suppose in order to derive a contradiction that $w_{[1]}$ ends before $\longueur{UU} + tr(\dsuu, x)\lvert \generateur{u} \rvert + \longueur{\prefix{u}} - \lvert \lcs (\generateur{u}, r) \rvert + 1 = \longueur{UU} + (e_1-k)\longueur{\generateur{u}} + \longueur{\prefix{u}} - \longueur{\lcs (\prefix{u}\suffix{u},\suffix{u}\prefix{u})} +1$, hence $w_{[1]}$ is a factor of $\generateur{u}^{e_1+e_2}\prefix{u}\generateur{u}^{e_1+e_2-k}$, $\longueur{w}>\longueur{\generateur{u}}$ (by definition of the $\epsilon$-mates) and, by synchronization principle,  we are left with three possibilities: either $w = \generateur{\conj{u}}^n$ for a certain $n$, or $w = \generateur{\conj{u}}^{l} \prefix{\conj{u}}$ for a certain $l$ or $w = \generateur{\conj{u}}^{l} \prefix{\conj{u}}\generateur{\conj{u}}^{m}$ for a certain $l$ and a certain $m$..

\begin{itemize}
\item If $w = \generateur{\conj{u}}^n$ for a certain $n$, then $w^2$ is part of the factor $\generateur{u}^{e_1+e_2-k}$ that start at position $\generateur{u}^{e_1}\prefix{u}\generateur{u}^{e_2}\generateur{u}^{e_1}\prefix{u}+1$, and $w^2$ is repeated further in $x$.
\item If $w = \generateur{\conj{u}}^{l} \prefix{\conj{u}}$ for a certain $l$, by synchronization principle \ref{fw}, $\longueur{\prefix{\conj{u}}}=\longueur{\prefix{u}}$ and $w_{[1]}$ ends in the second core of the interrupt of \dsuu, hence $w^2$ is repeated further (namely at position $s(w_{[1]})+(e_1+e_2-k)\longueur{\generateur{u}}+\longueur{\prefix{u}}$).
\item If $w = \generateur{\conj{u}}^{l} \prefix{\conj{u}}\generateur{\conj{u}}^{m}$ for a certain $l$ and a certain $m$ (note that $l$ and $m$ are fixed), as for $\beta$-mates, we will study the canonical form of $w$, that is when $w = \generateur{u}^l\prefix{u}\generateur{u}^m$. The results apply to the conjugates of \dsuw. Write:
\[  x'' = \generateur{u}^{e_1}\prefix{u}\generateur{u}^{e_1+e_2}\prefix{u}\generateur{u}^{e_1+e_2-k}\prefix{u}\generateur{u}^{e_1+e_2}\prefix{u}\generateur{u}^{e_2+t-k}, \]
\[ x'' = \generateur{u}^{e_1}\prefix{u}\generateur{u}^{e_1+e_2-l}\underbrace{\generateur{u}^l\prefix{u}\generateur{u}^{e_1+e_2-k-l}}_{w_{[1]}}\underbrace{\generateur{u}^l\prefix{u}\generateur{u}^{e_1+e_2-k-l}}_{w_{[2]}}\generateur{u}^{k+l}\prefix{u}\generateur{u}^{e_2+t-k}. \]
It follows that $m=e_1+e_2-k-l$. Because $W$ has $\generateur{u}^l\prefix{u}$ as a prefix, $W_{[2]}$ can only start at position $e(w_{[1]})+1$ (but then $\longueur{w}=\longueur{W}$, a contradiction), or at position $\longueur{\generateur{u}^{e_1}\prefix{u}\generateur{u}^{e_1+e_2}\prefix{u}\generateur{u}^{e_1+e_2-k}\prefix{u}\generateur{u}^{e_1+e_2-l}}+1$, that is after $e(w_{[1]})$, a contradiction.
\end{itemize}

Now, $T(\dsuu, \dsuw) \geq (2(e_1+e_2)-k)\longueur{\generateur{u}} + 2\longueur{\prefix{u}} - \longueur{\lcs (\prefix{u}\suffix{u},\suffix{u}\prefix{u})}$, $e_1-k > t$ by lemma \ref{gamas}, $\longueur{\generateur{u}}> \longueur{\lcs (\prefix{u}\suffix{u},\suffix{u}\prefix{u})}$ by lemma \ref{dft} and $(2(e_1+e_2)-k-1)\longueur{\generateur{u}} \geq 2(e_2+t) \longueur{\generateur{u}}$, hence $T(\dsuu, \dsuw) \geq 2(t+1)\longueur{\generateur{u}}$ while $d(\dsuu, \dsuw) \leq (t+1)\longueur{\generateur{u}}$.
\end{proof}

\begin{property}\label{etas2}
Let $x$ be a word starting with a FS double square \dsuu. If \dsuu\ has an $\eta$-mate \dsuv\ then \dsuv's first exponent is 1,  $\generateur{v}=\generateur{u}^{n-1}\prefix{u}$ and $\prefix{v}=\suffix{u}$ for some integer $n\geq2$.
\end{property}

\begin{proof}
By definitions of the $\eta$-mates, $v=\conj{\generateur{u}}^n$ for $n\geq2$, and by definition of the FS double squares, $v=\generateur{v}^{f_1}\prefix{v}$ for a primitive word $\generateur{v}$ and $\prefix{v}$ a prefix of $\generateur{v}$. If $f_1\geq2$, then $\conj{\generateur{u}}^2$ and $\generateur{v}^{f_1+1}$ have a common factor of length $\longueur{\generateur{u}}+ \longueur{\generateur{v}}$ and, by periodicity lemma \ref{fw}, $\generateur{u}$ is not primitive, a contradiction.\\
Also because $v=\conj{\generateur{u}}^n$ for $n\geq2$, we have  $V=\conj{\generateur{u}}^{n+k}\prefix{u}'$ for an integer $k<n$ and some proper prefix $\prefix{u}'$ of $\generateur{u}$. Since $V^2$ starts within the factor $\generateur{u}^{e_1+e_2}$ in $x=\generateur{u}^{e_1}\prefix{u}\generateur{u}^{e_1+e_2}\prefix{u}\generateur{u}^{e_2}r$, by synchronization principle, $\prefix{u}'=\prefix{u}$. It follows that $\generateur{v}=\generateur{u}^{n-1-k}\prefix{u}, \prefix{v}=\suffix{u}\generateur{u}^{k}$ for some integer $k$  and by synchronization principle, $k=0$, (note that $\generateur{u}$ has $\suffix{u}$ as a prefix).\\


\end{proof}

\begin{corollary}\label{prefsuf}
Let $x$ be a word starting with a FS double square \dsuu. If \dsuu\ has an $\eta$-mate, then $\suffix{u}$ is a prefix of $\generateur{u}$.
\end{corollary}

\begin{proof}
By property \ref{etas2}, $\generateur{v}=\generateur{u}^{n-1}\prefix{u}$ and $\prefix{v}=\suffix{u}$. By lemma \ref{lams}, $\prefix{v}$ is a prefix of $\generateur{v}$, hence $\suffix{u}$ is a prefix of $\generateur{u}$
\end{proof}

\begin{corollary}\label{etas4}
Let $x$ be a word starting with a FS double square \dsuu. If \dsuu\ has an $\eta$-mate \dsuv\ then $\longueur{\lcp(\prefix{v}\suffix{v},\suffix{v}\prefix{v})} = \longueur{\lcp(\prefix{u}\suffix{u},\suffix{u}\prefix{u})}-\longueur{\generateur{u}}+\longueur{\prefix{u}}$ and $\lcs(\prefix{v}\suffix{v},\suffix{v}\prefix{v}) = \lcs(\prefix{u}\suffix{u},\suffix{u}\prefix{u})$.
\end{corollary}

\begin{proof}
By property \ref{etas2}, $\generateur{v}=\generateur{u}^{n-1}\prefix{u}$, $\prefix{v}=\suffix{u}$ and by corollary \ref{prefsuf}, $\suffix{u}$ is a prefix of $\generateur{u}$, hence $\generateur{u}=\suffix{u}\suffix{u'}$ for the suffix $\suffix{u'}$ of $\generateur{u}$ of length $\longueur{\prefix{u}}$. It follows that $\prefix{v} = \suffix{u}, \suffix{v}=\suffix{u'}\generateur{u}^{n-2}\prefix{u}$ and:
\begin{align*}
\lcs (\prefix{v}\suffix{v}, \suffix{v}\prefix{v})&=\lcs (\generateur{u}^{n-1}\prefix{u}, \suffix{u'}\generateur{u}^{n-1}) = \lcs (\prefix{u}(\suffix{u}\prefix{u})^{n-1}, \suffix{u'}(\prefix{u}\suffix{u})^{n-1}) \\ &=  \lcs (\prefix{u}\suffix{u}, \suffix{u}\prefix{u}),\\
\lcp (\prefix{v}\suffix{v}, \suffix{v}\prefix{v})&=\lcp (\generateur{u}^{n-1}\prefix{u}, \suffix{u'}\generateur{u}^{n-1}) = \lcp(\suffix{u'},\prefix{u}). 
\end{align*} 
Now, $\lcp(\prefix{u}\suffix{u}, \suffix{u}\prefix{u})= \lcp(\suffix{u}\suffix{u'}, \suffix{u}\prefix{u})$ and since $\longueur{\lcp(\prefix{u}\suffix{u}, \suffix{u}\prefix{u})}<\longueur{\generateur{u}}$, $\lcp(\suffix{u}\suffix{u'}, \suffix{u}\prefix{u}) = \suffix{u}\lcp(\suffix{u'},\prefix{u})$, and $\longueur{\lcp (\prefix{v}\suffix{v}, \suffix{v}\prefix{v})} = \longueur{\lcp(\prefix{u}\suffix{u}, \suffix{u}\prefix{u})}-\longueur{\suffix{u}}$ while $\longueur{\suffix{u}}=\longueur{\generateur{u}}-\longueur{\prefix{u}}$.
\end{proof}



\begin{corollary}\label{etas5}
Let $x$ be a word starting with a FS double square \dsuu. If \dsuu\ has an $\eta$-mate \dsuv, then \dsuu's first exponent is strictly greater than two.
\end{corollary}

\begin{proof}
Because \dsuu\ has an $\eta$-mate, by property \ref{etas2}, \[x= \generateur{u}^{e_1}\prefix{u}\generateur{u}^{e_1+e_2}\prefix{u}\generateur{u}^{2n-1}\prefix{u}r'\]
 for an integer $n\geq2$ and a suffix $r'$ of $x$ and $\longueur{\lcp(r', \generateur{u})}< \longueur{\suffix{u}}$. If $e_1\leq2$ then $u^2$ has a further occurrence in $x$, namely in the factor $\generateur{u}^{e_1+e_2}\prefix{u}\generateur{u}^{2n-1}\prefix{u}$ and \dsuu\ is not a FS double square. \\
\end{proof}

\begin{corollary}\label{etas3}
Let $x$ be a word starting with a FS double square \dsuu. If \dsuu\ has an $\eta$-mate \dsuv\ then \dsuv\ cannot have $\beta$-mates.
\end{corollary}

\begin{proof}
By property \ref{etas2}, \dsuv's first exponent is 1 and by property \ref{betas}, \dsuv\ cannot have $\beta$-mates.
\end{proof}

\begin{property}\label{etassimple}
Let $x$ be a word starting with a FS double square \dsuu. The number of $\eta$-mates of \dsuu\ is at most $\longueur{p}+\longueur{s}-\longueur{\generateur{u}}+\longueur{\prefix{u}}$. Moreover, if \dsuu\ has an $\eta$-mate, $\longueur{x} \geq \longueur{Uu} + 3\longueur{\generateur{u}}+\longueur{\prefix{u}}$.
\end{property}

\begin{proof}
By corollaries \ref{etas4} and \ref{etas3}, \dsuv\ has at most $\longueur{p}+\longueur{s}-\longueur{\suffix{u}}=\longueur{p}+\longueur{s}-\longueur{\generateur{u}}+\longueur{\prefix{u}}$ shifts.\\ 
By property \ref{etas2}, $x=\generateur{u}^{e_1}\prefix{u}\generateur{u}^{e_1+e_2}\prefix{u}\generateur{u}^{2n-1}\prefix{u}r$ for some suffix $r$ of $x$ and the integer $n\geq2$. If follows that $\longueur{x} \geq \longueur{Uu} + 3\longueur{\generateur{u}}+\longueur{\prefix{u}}$.
\end{proof}

\begin{property}\label{gammaeta}
Let $x$ be a word starting with a FS double square \dsuu. If \dsuu\ has a $\gamma$-mate, then \dsuu\ cannot have $\eta$-mates.
\end{property}

\begin{proof}
As in the proof \ref{gamas}, if \dsuu\ has a $\gamma$-mate, then $x$ has the prefix:
\[ x'' = \generateur{u}^{e_1}\prefix{u}\generateur{u}^{e_2}\generateur{u}^{e_1}\prefix{u}\generateur{u}^{e_2}\generateur{u}^{e_1-k}\prefix{u}\generateur{u}^{e_1+e_2}\prefix{u}\generateur{u}^{e_2+t-k}, \]
for fixed $t$ and $k$, and any $\eta$-mate would have its small repetition repeated later in $x$ (namely in the factor $\generateur{u}^{e_1+e_2}\prefix{u}\generateur{u}$ at that starts at position \\$\longueur{ \generateur{u}^{e_1}\prefix{u}\generateur{u}^{e_2}\generateur{u}^{e_1}\prefix{u}\generateur{u}^{e_2}\generateur{u}^{e_1-k}\prefix{u}}+1$).
\end{proof}

\begin{property}\label{deltaeta}
Let $x$ be a word starting with a FS double square \dsuu. If \dsuu\ has a $\delta$-mate \dsuv\ and an $\eta$-mate \dsuw , then $T(\dsuu, \dsuv) \geq (e_1+e_2+2)\longueur{\generateur{u}} + 2 \longueur{\prefix{u}}$.
\end{property}

\begin{proof}
Because $x$ has $\eta$-mates, by property \ref{etas2}, $x$ has the prefix:
\[x'=\generateur{u}^{e_1}\prefix{u}\generateur{u}^{e_1+e_2}\prefix{u}\generateur{u}^{2n-1}\prefix{u}r'\] for a factor $r', \longueur{\lcp(r',\suffix{u})}<\longueur{\suffix{u}}$, hence $\suffix{u}$ is not a prefix of $r'$, and by corollary \ref{prefsuf}, $\generateur{u}$ is not a prefix of $r'$ either, and $tr(\dsuu, w)= 2n-1-e_2$.\\
Now, \dsuv\ is a $\delta$-mate of \dsuu: $\longueur{v}>\longueur{U}$. By property \ref{deltas}, $T(\dsuu, \dsuv) \geq \lvert U \rvert$. Suppose that $T(\dsuu, \dsuv)\leq\longueur{\generateur{u}^{e_1+e_2}\prefix{u}\generateur{u}^{2n-2}\prefix{u}}$ for $n \geq 2$, then $v_{[2]}$ contains the factor $\generateur{u}\prefix{u}r''$ for $r''$ a prefix of $r'$, $\longueur{r''} = \longueur{\generateur{u}}$, $r''\neq\suffix{u}, r''\neq\generateur{u}$ while $v_{[1]}$ is a factor of $\generateur{u}^{e_1}\prefix{u}\generateur{u}^{e_1+e_2}\prefix{u}\generateur{u}^{2n-2}$ contradicting the synchronization principle.
\end{proof}

\begin{property}\label{zetaeta}
Let $x$ be a word starting with a FS double square \dsuu\ with an $\eta$-mate \dsuv\ (that is, $\longueur{v}=n\longueur{\generateur{u}}$ for some $n\geq2$). If \dsuu\ has a $\zeta$-mate \dsuw\, then $T(\dsuu, \dsuw)\geq (e_1+e_2+n-1)\longueur{\generateur{u}}+2\longueur{\prefix{u}}-\longueur{s}$.
\end{property}

\begin{proof}
Because \dsuu\ has an $\eta$-mates \dsuv, by property \ref{etas2}, $x$ has the prefix:
\[x'=\generateur{u}^{e_1}\prefix{u}\generateur{u}^{e_1+e_2}\prefix{u}\generateur{u}^{2n-1}\prefix{u}r'\] for a certain $n\geq2$, with $\lcp(\generateur{u},r')\neq \prefix{u}$ and $\lcp(\generateur{u},r')\neq \suffix{u}$.\\
Suppose that $T(\dsuu,\dsuw) < (e_1+e_2+n)\longueur{\generateur{u}}+2\longueur{\prefix{u}}$, hence $w_{[1]}$ is a factor of $\generateur{u}^{e_1+e_2}\prefix{u}\generateur{u}^{n}\prefix{u}$ and, by definition of the $\zeta$-mates, $\longueur{w}>\longueur{\generateur{u}}$.\\
If $\exists m \geq 2, \longueur{w} = m\longueur{\generateur{u}}$ and $T(\dsuu, \dsuw) < (e_1+e_2+n)\longueur{\generateur{u}}+2\longueur{\prefix{u}}$, then, by definition of the $\eta$-mates,  $e(w_{[1]}) > \longueur{u\generateur{u}^{e_1+e_2-1}p}$. 
\begin{itemize}
\item If $e(w_{[1]}) < \longueur{u\generateur{u}^{e_1+e_2}\prefix{u}}-\longueur{s}$, then $w_{[2]}$ contains $w_2$, the factor of length $\longueur{\generateur{u}}$ that starts with the core of the interrupt of \dsuu\ while $w_{[1]}$ doesn't, a contradiction.
\item If $\longueur{u\generateur{u}^{e_1+e_2}\prefix{u}}-\longueur{s} \leq e(w_{[1]}) \leq \longueur{u\generateur{u}^{e_1+e_2}\prefix{u}}+\longueur{p}$, then $w_{[1]}$ has a certain conjugate $\conj{\generateur{u}}$ of $\generateur{u}$ as a prefix while $w_{[2]}$ starts with another one, a contradiction.
\item If $\longueur{u\generateur{u}^{e_1+e_2}\prefix{u}}+\longueur{p} < e(w_{[1]})$, then either $s(w_{[1]}) < \longueur{u\generateur{u}^{e_1+e_2-1}\prefix{u}}+\longueur{p}$ and $w_{[1]}$, or $ s(w_{[1]}) \geq \longueur{u\generateur{u}^{e_1+e_2-1}\prefix{u}}+\longueur{p}$:\\
If $s(w_{[1]}) < \longueur{u\generateur{u}^{e_1+e_2}\prefix{u}}-\longueur{s}$, $w_{[1]}$ contains $w_2$ the factor of length $\longueur{\generateur{u}}$ that starts with the core of the interrupt of \dsuu\ while $w_{[2]}$ doesn't, a contradiction. \\
If $s(w_{[1]}) \geq \longueur{u\generateur{u}^{e_1+e_2}\prefix{u}}-\longueur{s}$, then $w = \conj{\generateur{u}}^m$. If $T(\dsuu,\dsuw) < (e_1+e_2+n-1)\longueur{\generateur{u}}+2\longueur{\prefix{u}}-\longueur{s}$, then $m\leq n-1$ and $w^2$ has another occurrence in $x$ (namely at position $s(w)+\longueur{\generateur{u}}$), a contradiction, hence $T(\dsuu, \dsuw)\geq (e_1+e_2+n-1)\longueur{\generateur{u}}+2\longueur{\prefix{u}}-\longueur{s}$.

\end{itemize}
If, $\forall m,\longueur{w}\neq m\longueur{\generateur{u}}$, rewrite $x'$ as $\generateur{u}^{e_1}\prefix{u}\generateur{u}^{e_1+e_2}\prefix{u}\generateur{u}^{n}\prefix{u}(\suffix{u}\prefix{u})^{n-1}r'$ to notice that $w_{[1]}$ has to start before $\longueur{\generateur{u}^{e_1}\prefix{u}\generateur{u}^{e_1+e_2}\prefix{u}}-\longueur{s}$ (otherwise: $w^2$ and $\conj{\generateur{u}}^n$ have a common factor of length $\longueur{v}+\longueur{\generateur{u}}$ and the synchronization principle applies). By synchronization principle, we get that $w=\conj{\generateur{u}}^p\conj{\prefix{u}}\conj{\generateur{u}}^q$ for a conjugate $\conj{\generateur{u}}$ of $\generateur{u}$, $\conj{\prefix{u}}$ its prefix of length $\longueur{\prefix{u}}$ and the integers $p$ and $q$. Because, in $x'$, $\lcp(\generateur{u},r')\neq \prefix{u}$, we get that $q=0$. Now $W=\conj{\generateur{u}}^p\conj{\prefix{u}}\conj{\generateur{u}}^{q'}\conj{\prefix{u'}}$ for a prefix $\conj{\prefix{u'}}$ of $\conj{\generateur{u}}$ and an integer $q'$: by synchronization principle, $\conj{\prefix{u'}}=\emptyset$ and $\lcp(\generateur{u},r')=\generateur{u}$, a contradiction.
\end{proof}

\begin{property}\label{epsilonetabis}
Let $x$ be a word starting with a FS double square \dsuu. If \dsuu\ has a $\eta$-mate \dsuv\ then it has at most $2\longueur{\generateur{u}} - \longueur{s} - \longueur{p}$ $\epsilon$-mates.
\end{property}

\begin{proof}
Because \dsuu\ has $\eta$-mates, by property \ref{etas2}, $x$ has the prefix:\\
\[x'=\generateur{u}^{e_1}\prefix{u}\generateur{u}^{e_1+e_2}\prefix{u}\generateur{u}^{2n-1}\prefix{u}r'\]
for some $n\geq2$ hence there is a factor $\generateur{u}^{2n-1}\prefix{u}$ at position $\longueur{Uu}+1$, and any $\epsilon$-mate of \dsuu\ has to have $v^2$ ending after $\longueur{Uu}+\lcp(\prefix{u}\suffix{u}, \suffix{u}\prefix{u})+1$ (otherwise, $v^2$ is a factor of $\generateur{u}^3$), hence $v_{[1]}$ has to start after $\longueur{Uu}+\lcp(...) - 2\longueur{\generateur{u}}+1$ since $\longueur{v} \leq \longueur{\generateur{u}}$. For the same reason, $w_{[1]}$ has to start before $\longueur{Uu}-\lcs(\prefix{u}\suffix{u}, \suffix{u}\prefix{u})+1$. If an $\epsilon$-mate start at each of those positions, we have $2\longueur{\generateur{u}}-\longueur{s}-\longueur{p}$ $\epsilon$-mates.
\end{proof}

\begin{property}\label{betaetatous}
Let $x$ be a word starting with a FS double square \dsuu. If \dsuu\ has $\alpha$, $\beta$, $\epsilon$ and $\eta$-mates, then there is at most $(\frac{e_1-e_2-1}{2}+2)\longueur{\generateur{u}}+\longueur{\prefix{u}}-\longueur{s}$ of them.
\end{property}

\begin{proof}
By property \ref{alphas}, \dsuu\ has at most $\longueur{p}$ $\alpha$-mates.\\
Let's bound the number of $\beta$-mates. Write: 
\[x= \generateur{u}^{e_1}\prefix{u}\generateur{u}^{e_1+e_2}\prefix{u}\generateur{u}^{e_2}\generateur{u}^{t(\dsuu, x)}r\]
 for a suffix $r$ of $x$ and, because \dsuu\ has $\beta$-mates, by property \ref{betas2} , $\longueur{\lcp(r, \generateur{u})} < \longueur{\prefix{u}}$ if $e_1-\lfloor \frac{e_1-e_2}{2}\rfloor=e_2+t(\dsuu, x)$.\\
Because \dsuu\ has $\eta$-mates, by property \ref{etas2}, 
\[x= \generateur{u}^{e_1}\prefix{u}\generateur{u}^{e_1+e_2}\prefix{u}\generateur{u}^{2n-1}\prefix{u}r'\]
 for an integer $n\geq2$ and a suffix $r'$ of $x$ and $\longueur{\lcp(r', \generateur{u})}< \longueur{\suffix{u}}$ by corollary \ref{prefsuf}. \\
It follows that $t(\dsuu, x)= 2n-1-e_2$ and $r = \prefix{u}r'$.\\
If $t(\dsuu, x) < \lfloor \frac{e_1-e_2}{2} \rfloor$ then, by property \ref{betas}, the number of $\beta$-mates is less than $(\lfloor \frac{e_1-e_2}{2}\rfloor-2)(\lvert p \rvert + \lvert s \rvert + 1) + \lvert s \rvert + \min (\lvert \lcp(\generateur{u}, \prefix{u}r')\rvert +1$.\\
If $t(\dsuu,x)\geq e_1-\lfloor \frac{e_1-e_2}{2}\rfloor-e_2$ then the small repetitions of the $\beta$-mates that have the factorization $\generateur{\conj{u}^{e_1-\lfloor \frac{e_1-e_2}{2}\rfloor}}\prefix{\conj{u}}\generateur{\conj{u}}^{e_2+\lfloor \frac{e_1-e_2}{2}\rfloor}$ are repeated in the factor $\generateur{u}^{e_1+e_2}\prefix{u}\generateur{u}^{2n-1}\prefix{u}r'$ that starts at position $\longueur{u}+1$, and the number of $\beta$-mates is less than $(\lfloor \frac{e_1-e_2}{2}\rfloor-1)(\lvert p \rvert + \lvert s \rvert + 1)$.\\
Now, if $ \lfloor \frac{e_1-e_2}{2} \rfloor \leq t(\dsuu,x) < e_1-\lfloor \frac{e_1-e_2}{2}\rfloor-e_2$, that is $e_2+\lfloor \frac{e_1-e_2}{2} \rfloor \leq 2n-1< e_1-\lfloor \frac{e_1-e_2}{2}\rfloor$, which can only be solved if $e_1$ and $e_2$ have different parity, then $\lfloor \frac{e_1-e_2}{2}\rfloor = \frac{e_1-e_2-1}{2}$ and the number of $\beta$-mates is less than $\frac{e_1-e_2-1}{2}(\longueur{p} + \longueur{s} +1)$.\\
By property \ref{etassimple}, \dsuu\ has at most $\longueur{p}+\longueur{s}-\longueur{\generateur{u}}+\longueur{\prefix{u}}$ $\eta$-mates.\\
Because \dsuu\ has $\eta$-mates, by property \ref{epsilonetabis}, \dsuu\ has at most $2\longueur{\generateur{u}} - \longueur{s} - \longueur{p}$ $\epsilon$-mates.\\
Summing all of the above leads to at most $(\frac{e_1-e_2-1}{2})(\longueur{p}+\longueur{s})+\longueur{\generateur{u}}+\longueur{\prefix{u}}+\longueur{p}\leq (\frac{e_1-e_2-1}{2}+2)\longueur{\generateur{u}}+\longueur{\prefix{u}}-\longueur{s} $ $\alpha$, $\beta$, $\epsilon$ and $\eta$-mates.
\end{proof}

\section{Conclusion}
Fraenkel and Simpson conjectured that a word of length $n$ has less than $n$ distinct squares. That conjecture is motivated by the existence of a family of words that asymptotically reach that ratio of distinct squares over number of letters of 1: the FS words presented in \cite{FS1}. \\
The FS word $Q_j$ is defined as $Q_j=q_1q_2...q_j$ for $q_i=0^{i+1}10^i10^{i+1}1$.\\

\begin{figure}[H]
\begin{align*}
q_{i}.q_{i+1} = 0^{i+1} 1 0^{i} 1 \underbrace{0^{i+1} 1. 0^{{i+2}} 1 0^{i+1}}_{(0^{{i+1}-t}10^{t+1})^2} 1 0^{i+2} 1.
\end{align*}
  \caption{Squares with 2 ones of $Q_j$.}
  \label{deux}
\end{figure}

Note that there are FS double squares in $Q_j$. In Figures \ref{deux} et \ref{quatre}, the squares $u^2 = (0^{j-t}10^{t+1})^2$ and $U^2 = (0^{i+1-k}10^{i+2}10^k)^2$, for $j=i+1, t=k\leq j$, start at the same position and the canonical factorization of that double square \dsuu\ is ${u_{i,t}}_0 = 0^{i+1-t}10^t, {u_{i,t}}_1 = 0, e_1 = e_2 = 1$.

\begin{figure}[H]
 \[  q_i.q_{i+1}.q_{i+2} = 0^{i+1} 1 
 \rlap{$\overbrace{\phantom{ 0^i 1. 0^{i+1} 1.  0^{i+2} 1 0^{i+1} 1 0^{i+.}}}^{(0^{i-k}10^{i+1}10^{2+k})^2}$}
  0^i 1 \underbrace{0^{i+1} 1.  0^{i+2} 1 0^{i+1} 1 0^{i+2} 1.  0^{i+3}}_{(0^{i+1-k}10^{i+2}10^k)^2} 1 0^{i+2} 1 0^{i+3} 1.\] 
  \caption{Squares with 4 ones of $Q_j$.}
  \label{quatre}
\end{figure}

We provided a proof that a word of length $n$ has less than $\frac{3n}{2}$ based on a bound on the number of FS double squares in a word. The presence of FS double squares in FS words proves the importance of understanding them for that problem. The recent proof for the runs conjecture in \cite{bannai1}, and a proof that the words with the maximal ratio of squares over length are binary in \cite{MS1}, provide new tools to tackle that problem.

\section{Bibliography}
\newcommand{\noopsort}[1]{} \newcommand{\printfirst}[2]{#1}
  \newcommand{\singleletter}[1]{#1} \newcommand{\switchargs}[2]{#2#1}

\end{document}